\documentclass{amsart}
\usepackage[utf8]{inputenc}
\usepackage{mathtools}
\usepackage{amsthm}
\usepackage{amssymb}
\usepackage{stmaryrd}
\usepackage[
backend=biber,
style=alphabetic,
maxbibnames=99,
doi=false
]{biblatex}
\usepackage{bbm} 
\usepackage{tikz-cd}
\usepackage{graphicx}  
\usepackage{graphics}
\usepackage{upgreek}
\usepackage{mathrsfs}
\usepackage{appendix}
\usepackage[hidelinks]{hyperref}
\usepackage{dutchcal}
\usepackage{verbatim}
\bibliography{main.bib}
\usepackage{spectralsequences}
\usepackage[shortlabels]{enumitem}
\usepackage{bbm}

\newtheorem{theorema}{Theorem}

\newtheorem{thm}{Theorem}[section]
\newtheorem*{thm*}{Theorem}
\newtheorem{prop}[thm]{Proposition}
\newtheorem{lemma}[thm]{Lemma}

\newtheorem{cor}[thm]{Corollary}

\newtheorem{remark}[thm]{Remark}
\theoremstyle{definition}
\newtheorem{mydef}[thm]{Definition}
\newtheorem*{mydef*}{Definition}
\newtheorem{ex}[thm]{Example}

\newcounter{exCnt}
\newcounter{queCnt}

\newcommand{\Z}{\mathbb{Z}}

\newcommand{\Co}{\mathcal{C}}

\newcommand{\set}[2]{\ifthenelse{\equal{#2}{}}{\{ #1 \}}{\{ #1 \mid #2 \}}}

\DeclareMathOperator{\Ext}{Ext}

\DeclareMathOperator{\Hom}{Hom}

\DeclareMathOperator{\HH}{HH}

\DeclareMathOperator{\ide}{id}

\DeclareMathOperator{\Ca}{Ca}
\DeclareMathOperator{\op}{op}
\DeclareMathOperator{\opp}{op}

\DeclareMathOperator{\ev}{ev}

\DeclareMathOperator{\Bicomod}{\text{-}\mathbf{Bicomod}}
\DeclareMathOperator{\Bimod}{\text{-}\mathbf{Bimod}}

\DeclareMathOperator{\Cob}{Cob}
\DeclareMathOperator{\Hh}{H}

\DeclareMathOperator{\EndOp}{\underline{End}}
\DeclareMathOperator{\CoEndOp}{\underline{Coend}}

\DeclareMathOperator{\tr}{tr}

\newcommand\sbu[1][.5]{\mathbin{\vcenter{\hbox{\scalebox{#1}{$\bullet$}}}}}

\newcommand{\ld}[1]{\prescript{*}{}{#1}}

\usetikzlibrary{decorations.pathmorphing,decorations.pathreplacing,angles,quotes,calc}



\AtEveryBibitem{\clearfield{issn}}
\AtEveryCitekey{\clearfield{issn}}

\title{Corings, their dual rings and relative (co)Hochschild cohomology}
\author{Jonathan Lindell}

\begin{document}
\maketitle

\begin{abstract}
  We show for a coring which is finitely generated projective as a left module that the Cartier cohomology is isomorphic to the relative Hochschild cohomology of the right algebra. Furthermore, we show that this isomorphism lifts to the level of \(B_{\infty}\)-algebras of the chain complexes, by showing that the opposite \(B_{\infty}\)-algebra of the relative Hochschild cochains of the right algebra is isomorphic to the \(B_{\infty}\)-algebra of Cartier cochains. Lastly, we apply this to entwining structures where the coalgebra is finite-dimensional, to get a description of the equivariant cohomology of the entwining structure as the relative Hochschild cohomology of the twisted convolution algebra. 
\end{abstract}

\section{Introduction}
\label{sec:introduction}

Hochschild cohomology was first defined by Hochschild in \cite{Hochschild1945}. It was shown by Gerstenhaber in the seminal paper \cite{Gerstenhaber1963} that Hochschild cohomology has the structure of a Gerstenhaber algebra, enconding the deformation theory of the algebra. Dually, Cartier defined in \cite{Cartier1955} a natural cohomology theory for coalgebras using the Ext-functor in the category of bicomodules, called Cartier cohomology or coHochschild cohomology. This was later generalised to corings in \cite{Guzman1989}, which similarly has the structure of a Gerstenhaber algebra, see 30.8 in \cite{CoringsBW}. In \cite{Hochschild1956}, Hochschild defined relative homological algebra and relative Hochschild cohomology using the relative Ext-functor, taking into account the presence of a subalgebra. Gerstenhaber and Schack used relative Hochschild cohomology in the context of deformation theory, see \cite{GerstenhaberSchack1986}. Recent developments \cite{CLMSS2019, CLMS2020, CibilsLanzilottaMarcosSolotar2024, FaitgGainutdinovSchweigert2024}, in the context of adding and deleting arrows to a quiver, Han's conjecture and deformations of tensor categories, has created a growing interest in relative homological algebra and relative Hochschild (co)homology. 

\hfill

Corings were first defined by Sweedler in \cite{Sweedler1975}, and independently by Roiter in \cite{Roiter1979} under the name of bocses in the context of matrix problems. Corings appear naturally in the context of non-commutative descent theory of modules \cite[Chapter 4]{CoringsBW}, entwining structures \cite[Chapter 5]{CoringsBW} and quasi-hereditary algebras \cite{KKO2014} among other applications. Associated to any \(B\)-coring \(\mathcal{C}\), there is the right algebra, also called the opposite of the left dual algebra, see \cite{Sweedler1975} and Section 1 in \cite{BurtButler1991}, which has \(B\) as a canonical subalgebra. In this paper, we study the connection between the Cartier cohomology of a coring and the relative Hochschild cohomology of the associated right algebra. We show the following

\begin{theorema}[Proposition~\ref{prop:extduality}, Theorem~\ref{thm:extduality}]\label{thm:A}
   Let \(\mathcal{C}\) be a \(B\)-coring. Then there is a natural morphism
  \begin{equation*}
    \Ext^{*}_{(\mathcal{C}\text{-}\mathcal{C}|B\text{-}B)}(M,N) \rightarrow \Ext^*_{(A^e|B^e)}(F_R(N), F_R(M))
  \end{equation*}
  of \(k\)-vector spaces. Moreover, if \(M\) and \(N\) are left finitely generated projective as \(B\)-modules, then the map above is an isomorphism. 
\end{theorema}

For \(M = N = \mathcal{C}\), we show that this lifts to the \(B_{\infty}\)-level. We use \(-^{\opp}\) to denote the opposite of \(B_{\infty}\)-algebras as defined in \cite{ChenLiWang2021} and for Gerstenhaber algebras as defined in Section 2.2.

\begin{theorema}[Theorem~\ref{thm:coringtoalg}]\label{thm:B}
  Let \(\mathcal{C}\) be a coring and let \(R^{\op}\) be the opposite of the right algebra for \(\mathcal{C}\). Then there is a \(B_{\infty}\)-morphism 
  \begin{equation*}
    C_{\Ca}^*(\mathcal{C}) \rightarrow C^*(R^{\op}| B^{\op}) \cong C^*(R|B)^{\opp}
  \end{equation*}
  from the \(B_{\infty}\)-algebra given by the Cartier cochain complex to the opposite \(B_{\infty}\)-algebra given by the relative Hochschild cochain complex, which thus induces an morphism of Gerstenhaber algebras
  \begin{equation*}
    \Hh_{\Ca}^*(\mathcal{C}) \rightarrow \HH^*(A|B)^{\opp}.
  \end{equation*}
  Moreover, if \(\mathcal{C}\) is left finitely generated as \(B\)-module, then the maps above are isomorphisms. 
\end{theorema}

As an application we get that the one-to-one correspondence shown in \cite{BKK2020} between quasi-hereditary algebras with homological exact Borel subalgebras and directed corings is compatible with Cartier cohomology and relative Hochschild cohomology.

\hfill

Entwining structures were first defined by Majid and Brezi\'nski \cite{BrzezinskiMajid1998} in the context of non-commmutative geometry. Moreover, it turned out that entwining structures unify several structures such as Hopf modules, relative Hopf modules, Yetter-Drinfeld modules and Doi-Hopf modules, see Chapter 5 in \cite{CoringsBW}. Given an entwining structure \((A,C,\psi)\), Brzezi\'nski defined the equivariant cohomology \(\Hh_{\psi-e}(C)\), see Section 5 in \cite{Brzezinski2001} for the dual version and Chapter 5 in \cite{CoringsBW}. As an application we prove that
\begin{theorema}\label{thm:C}
  Let \((A,C,\psi)\) be an entwining structure such that \(C\) is finite-dimensional. Then there is an isomorphism of \(B_{\infty}\)-algebras
  \begin{equation*}
    C^*_{\psi-e}(C) \cong C^*(\Hom_{\psi}(C,A)|A)^{\opp} \cong C^*((C^{\vee})^{\op} \# A | A)^{\op}
  \end{equation*}
  which thus induces an isomorphism of Gerstenhaber algebras
  \begin{equation*}
    \Hh^*_{\psi-e}(C) \cong \HH^*(\Hom_{\psi}(C,A)|A)^{\opp} \cong  \HH^*((C^{\vee})^{\op} \# A | A)^{\op}.
  \end{equation*}
\end{theorema}

As a corollary we obtain, in the case where the coalgebra is finite-dimensional, that the equivariant cohomology controls the deformation theory of the \(\psi\)-twisted convolution algebra \(\Hom_{\psi}(C,A)\), see Corollary~\ref{cor:deformation}. Lastly, we compute for trivial entwining structures the \(\psi\)-equivariant cohomology.  

\subsection*{Outline}
\label{sec:outline}

We begin by recalling \(A_{\infty}\)- and \(B_{\infty}\)-algebras, operads, relative Hochschild cohomology, Cartier cohomology and relative deformation theory in Section 2. In Section 3 we describe the connection between relative Ext on the category of bicomodules and the category of bimodules over the right algebra, and prove Theorem~\ref{thm:A}. In Section 4 we show that this lifts to the \(B_{\infty}\)-level and prove Theorem~\ref{thm:B}. Lastly, in Section 5, we recall entwining structures and prove Theorem~\ref{thm:C}. 

\subsection*{Acknowledgements.} I would like to thank Julian K\"ulshammer for his support and for making several important suggestions. I would also like to thank Lleonard Rubio y Degrassi for his support and several important discussions.

\section{Preliminaries}
\label{sec:background}

We start by recalling the definitions of and basic results for \(A_{\infty}\)-algebras, \(B_{\infty}\)-algebras and operads. We mainly follow the exposition in \cite{ChenLiWang2021} for \(A_{\infty}\)- and \(B_{\infty}\)-algebras and the exposition in \cite{GerstenhaberVoronov1995} for operads and their connection to \(B_{\infty}\)-algebras. Throughout, let \(k\) be a field.

\subsection{\(A_{\infty}\)-algebras and morphisms}

We recall the definition of and basic results for \(A_{\infty}\)-algebras.  

\begin{mydef}
  Let \(A = \oplus_{n \in \Z}A^n\) be a \(\Z\)-graded vector space. We say that \((A^n,m_n)_{n \in \Z}\) is an \(A_{\infty}\)\textit{-algebra} if there are \(k\)-linear graded maps
  \begin{equation*}
    m_n : A^{\otimes n} \rightarrow A
  \end{equation*}
  of degree \(2-n\) that satisfy the following relations
  \begin{equation*}
    \sum_{j=0}^{n-1}\sum_{s=1}^{n-j}(-1)^{j+s(n-j-s)}m_{n-s+1}(\ide^{\otimes j} \otimes m_s \otimes \ide^{\otimes n-j-s}) = 0
  \end{equation*}
  for \(n \geq 1\). Note in particular that this means that \(m_1\) is a differential. We say that \(A\) is \textit{strictly unital} if there exists an element \(1 \in A^0\) such that \(m_2(1,a) = m_2(a,1) = a\) for all \(a \in A\).  
\end{mydef}

Let \(A = \oplus_{p \in \Z} A^p\) be a graded vector space, then we define the \textit{suspension} of \(A\) to be the graded vector space \(sA\), where \(sA^p = A^{p+1}\).
\begin{mydef}
  Let \(A\) be a graded vector space. The \textit{tensor coalgebra} is the graded coalgebra
  \begin{equation*}
    T^c_k(A) := k \oplus A \oplus A^{\otimes 2} \oplus ...
  \end{equation*}
  where the comultiplication is given by deconcatination and the counit is given by projection onto \(k\). 
\end{mydef}

\begin{lemma}\label{lemma:lifting}\hfill
  \begin{enumerate}
  \item Let \(f : V^{\otimes n} \rightarrow V\) be a map of graded vector spaces of degree \(2-n\). This can be viewed as a map \(f: T^c(V) \rightarrow V\) by letting \(f\) be zero on all other components. Then the map \(f\) lifts to a coderivation \(\widetilde{f} : T^c(V) \rightarrow T^c(V)\) by
    \begin{equation*}
      \widetilde{f}(v_1,...,v_k) =
      \begin{cases}
        0 &\text{ if } k < n \\
        \sum_{i=0}^{k-n} (-1)^{(n-1)\cdot (|v_1| + ... + |v_i|)}(v_1,...,f(v_{i+1},...,v_{i+n}),...,v_k) &\text{ otherwise.}
      \end{cases}
    \end{equation*}
  \item There is a one-to-one correspondence between coderivations
    \begin{equation*}
      \sigma : T^c(V) \rightarrow T^c(V)
    \end{equation*}
    and systems of maps \(\{f_i : V^{\otimes i} \rightarrow V\}_{i \geq 0}\) given by \(\sigma = \sum \widetilde{f_i}\). 
  \end{enumerate}
\end{lemma}

For a proof see Lemma 2.3 in \cite{Tradler2008}. From this it now follows that

\begin{prop}\label{prop:barAinfty}
  The following are equivalent
  \begin{enumerate}
  \item \((A,m_n)\) is an \(A_{\infty}\)-algebra;
  \item \(T^c(sA)\) is a coalgebra together with a differential of degree \(1\). 
  \end{enumerate}
\end{prop}

\begin{mydef}
  Let \((A,m_n)\) and \((A',m_n')\) be \(A_{\infty}\)-algebras. An \(A_{\infty}\)\textit{-morphism} \(f = (f_n)_{n \geq 1} : A \rightarrow A'\) is a collection of graded maps \(f_n : A^{\otimes n} \rightarrow A'\) of degree \(1-n\) such that
  \begin{equation*}
    \sum_{\substack{a + s + t = n \\ a,t \geq 0, s \geq 1}} (-1)^{a+st}f_{a+1+t}(\ide^{\otimes a} \otimes m_s \otimes \ide^{\otimes t}) = \sum_{\substack{r \geq 1 \\ i_1 + ... + i_r = n}} (-1)^{\varepsilon}m_r'(f_{i_1} \otimes ... \otimes f_{i_r})
  \end{equation*}
  for all \(n \geq 1\), where \(\varepsilon = (r-1)(i_1-1) + ... + 1\cdot(i_{r-1}-1)\). Note that \(f_1\) is especially a chain map by the above identities. An \(A_{\infty}\)-morphism is \textit{strict} if \(f_i = 0\) for \(i \neq 1\). We say that an \(A_{\infty}\)-morphism is a \textit{quasi-isomorphism} if \(f_1\) is a quasi-isomorphism of chain complexes. 
\end{mydef}

One can show, similar to Proposition~\ref{prop:barAinfty}, that \(A_{\infty}\)-morphisms \(f : A \rightarrow A'\) correspond to dg coalgebras morphisms \(T^c(sf) : T^c(sA) \rightarrow T^c(sA')\) and vice-versa.

\subsection{\(B_{\infty}\)-algebras and morphisms}

The notion of \(B_{\infty}\)-algebras was first defined in \cite{GetzlerJones1994}, abstracting the definition from the structure that Baues, in \cite{Baues1981}, gave on the normalised cochain complex of any simplicial set. 

\begin{mydef}
  Let \(A\) be an \(A_{\infty}\)-algebra and let \(T^c(sA)\) be the corresponding dg coalgebra. We say that \(A\) is a \(B_{\infty}\)\textit{-algebra} if \(T^c(sA)\) has the structure of a dg bialgebra. An \(A_{\infty}\)-morphism \(f : A \rightarrow A'\) of \(B_{\infty}\)-algebras is a \(B_{\infty}\)\textit{-morphism} if the associated morphism \(T^c(sf) : T^c(sA) \rightarrow T^c(sA')\) is a map of dg bialgebras. We say that a \(B_{\infty}\)-morphism is a \textit{quasi-isomorphism} if the underlying \(A_{\infty}\)-morphism is a quasi-isomorphism. We say that an \(B_{\infty}\)-morphism is \textit{strict} if the underlying \(A_{\infty}\)-morphism is strict. 
\end{mydef}

By results similar to Lemma~\ref{lemma:lifting}, this is equivalent to that there are homogeneous maps
\begin{equation*}
  \mu_{p,q} : A^{\otimes p} \otimes A^{\otimes q} \rightarrow A \quad p,q \geq 0
\end{equation*}
of degree \(1-p-q\) fulfilling certain unitality and associativity conditions, as well as fulfilling certain higher Leibniz rules together with the \(\{m_n\}_{n \geq 1}\). For a detailed description of these conditions see Section 5.2 in \cite{ChenLiWang2021}.

\begin{mydef}
  We say that a \(B_{\infty}\)-algebra \((B,m_1,m_2,..,\mu_{p,q})\) is a \textit{brace} \(B_{\infty}\)\textit{-algebra} if \(m_n = 0\) for \(n \geq 3\) and \(\mu_{p,q} = 0\) for \(p \geq 2\). 
\end{mydef}

For a brace \(B_{\infty}\)-algebra, we define the \textit{brace operations} to be
\begin{equation*}
  a\{b_1,...,b_p\} := (-1)^{p|a|+(p-1)|b_1|+(p-2)|b_2|+...+|b_p|}\mu_{1,p}(a \otimes b_1 \otimes ... \otimes b_p)
\end{equation*}
for any \(a,b_1,...,b_p \in A\). We will shorten \(a\{b_1,...,b_p\} = a \{b_{1,p}\}\).

\hfill

In the case of a brace \(B_{\infty}\)-algebra the conditions on the maps \(\mu_{p,q}\) become the following
\begin{enumerate}
\item Higher pre-Jacobi identity:
  \begin{align*}
    (a \{b_{1,p}\}) \{c_{1,q}\} = \sum (-1)^{\epsilon} &a \{c_{1,i_1},b_1 \{c_{i_1+1,i_1+l_1}\},c_{i_1+l+1,i_2},b_2 \{c_{i_2+1,i_2+l_2}\},...,c_{i_p},\\
                                &b_p \{c_{i_p+1,i_p+l_p}\}, c_{i_p+l_p+1,q}\} 
  \end{align*}
  where the sum is taken over all sequences of nonnegative integers \((i_1,...,i_p;l_1,...,l_p)\) such that
  \begin{equation*}
    0 \leq i_i \leq i_1+l_1 \leq i_2 \leq i_2 + l_2 \leq i_3 \leq ... \leq i_p + l_p \leq q
  \end{equation*}
  and
  \begin{equation*}
    \epsilon = \sum_{l=1}^p\left((|b_l|-1)\sum_{j=1}^{i_l}(|c_j|-1)\right).
  \end{equation*}
\item Distributivity:
  \begin{equation*}
    m_2(a_1 \otimes a_2)\{b_{1,q}\} = \sum_{j=0}^q(-1)^{|a_2|\sum_{i=1}^j(|b_i|-1)}m_2((a_1 \{b_{1,j}\}) \otimes (a_2 \{b_{j+1,q}\})).
  \end{equation*}
\item Higher homotopy:
  \begin{align*}
    &m_1(a \{b_{1,p}\}) - (-1)^{|a|(|b_1|-1)}m_2(b_1 \otimes (a \{b_{2,p}\})) + (-1)^{\epsilon_{p-1}}m_2((a \{b_{1,.p-1}\}) \otimes b_p) \\
    = &m_1(a) \{b_{1,p}\} - \sum_{i=0}^{p-1}(-1)^{\epsilon_i}a \{b_{1,i},m_1(b_{i+1}),b_{i+2,p}\} + \sum_{i=0}^{p-2}(-1)^{\epsilon_{i+1}}a \{b_{1,i},m_2(b_{i+1,i+2}),b_{i+3,p}\}
  \end{align*}
  where \(\epsilon_0 = |a|\) and \(\epsilon_i = |a| + \sum_{j=1}^{i}(|b_j|-1)\) for \(i \geq 1\). 
\end{enumerate}

All the \(B_{\infty}\)-algebras in this article will be brace \(B_{\infty}\)-algebras. Given a dg algebra map \(f : A \rightarrow A'\) between brace \(B_{\infty}\)-algebra the following characterisation of strict \(B_{\infty}\)-morphisms was given, see Lemma 5.15 in \cite{ChenLiWang2021}. 

\begin{lemma}\label{lemma:strictMorphism}
  Let \(f : A \rightarrow A'\) be a dg algebra map between two brace \(B_{\infty}\)-algebras. Then \(f\) is a strict \(B_{\infty}\)-morphism if and only if
  \begin{equation*}
    f(a \{b_1,...,b_p\}) = f(a) \{f(b_1),...,f(b_p)\}
  \end{equation*}
  for all \(p \geq 1\) and all \(a, b_1,...,b_p \in A\). 
\end{lemma}

\begin{mydef} 
  A \textit{Gerstenhaber algebra} \((H, \smile, [-,-])\) is a \(\Z\)-graded \(k\)-module \(H\) for which \((H,\smile)\) is a graded commutative associative algebra, \((sH,[-,-])\) is a graded Lie algebra with bracket \([-,-]\) and
  \begin{equation*}
    [\alpha \smile \beta, \gamma] = [\alpha, \gamma] \smile \beta + (-1)^{|\alpha|(|\gamma|-1)} \alpha \smile [\beta,\gamma]
  \end{equation*}
  for all homogeneous elements \(\alpha,\beta,\gamma \in H\). A morphism of Gerstenhaber algebras is a morphism of \(\Z\)-graded \(k\)-modules which is both a graded algebra morphism and graded Lie algebra morphism of the shifted algebra. We denote the category of Gerstenhaber algebras over \(k\) by \(\mathbf{Ger}_k\).
\end{mydef}

The following is shown in Section 5.2 in \cite{GetzlerJones1994}. 

\begin{prop}\label{prop:BgivesG}
  Let \((A,m_n,\mu_{pq})\) be a \(B_{\infty}\)-algebra. Then there is a natural Gerstenhaber structure on \((H^*(A,m_1), \smile, [-,-])\) on the cohomology, given by
  \begin{align*}
    f \smile g &= m_2(f,g). \\
    [f,g] &= (-1)^{|f|}\mu_{1,1}(f,g) - (-1)^{(|f|-1)(|g|-1)+|g|}\mu_{1,1}(g,f).
  \end{align*}
  Moreover, a \(B_{\infty}\)-quasi-isomorphism induces an isomorphism of Gerstenhaber algebras on the level of cohomology. 
\end{prop}

We recall two different natural operations on \(B_{\infty}\)-algebras given in Section 5.3 in \cite{ChenLiWang2021}. 

\begin{mydef}\label{def:opBinfinity}
  Let \((A,m_n,\mu_{pq})\) be a \(B_{\infty}\)-algebra. The \textit{opposite} \(B_{\infty}\)\textit{-algebra} \(A^{\opp}\) is the \(B_{\infty}\)-algebra \((A,m_n,\mu_{pq}^{\opp})\), where
  \begin{equation*}
    \mu_{pq}^{\opp}(a_1 \otimes ... \otimes a_p \otimes b_1 \otimes ... \otimes b_q) = (-1)^{pq+\varepsilon}\mu_{qp}(b_1 \otimes ... \otimes b_q \otimes a_1 \otimes ... \otimes a_p)
  \end{equation*}
  where \(\varepsilon := (|b_1| + ... + |b_q|)(|a_1| + ... + |a_p|)\). 
\end{mydef}

On the level of homology, this induces the following operation. 

\begin{mydef}
  Let \((H,\smile,[-,-])\) be a Gerstenhaber algebra. The \textit{opposite Gerstenhaber algebra} \(H^{\op}\) is the Gerstenhaber algebra \((H,\smile, [-,-]^{\op})\), where we define \([f,g]^{\op} = -[g,f]\).  
\end{mydef}

\begin{mydef}\label{def:trBinfinity}
  Let \((A,m_n,\mu_{pq})\) be a \(B_{\infty}\)-algebra. The \textit{transpose} \(B_{\infty}\)\textit{-algebra} \(A^{\tr}\) is the \(B_{\infty}\)-algebra \((A,m_n^{\tr},\mu_{pq}^{\tr})\), where
  \begin{align*}
    m_n^{\tr}(a_1 \otimes ... \otimes a_n) &:= (-1)^{\varepsilon_n} m_n(a_n \otimes ... \otimes a_1) \\
    \mu_{pq}^{\tr}(a_1 \otimes ... \otimes a_p \otimes b_1 \otimes ... \otimes b_q) &:= (-1)^{\varepsilon}\mu_{pq}(a_p \otimes ... \otimes a_1 \otimes b_q \otimes ... \otimes b_1)
  \end{align*}
  where
  \begin{align*}
    \varepsilon_n &:= \frac{(n+1)(n-2)}{2} + \sum_{j=1}^{n-1} |a_j|(|a_{j+1}| + ... + |a_n|) \\
    \varepsilon &:= 1 + \frac{p(p+1)}{2} + \frac{q(q+1)}{2} + \sum_{j=1}^{p-1}|a_j|(|a_{j+1}| + ... + |a_p|) + \sum_{j=1}^{q-1}|b_j|(|b_{j+1}| + ... + |b_q|).
  \end{align*}
\end{mydef}

The two notions turn out to be isomorphic as \(B_{\infty}\)-algebras.

\begin{thm}[Theorem 5.10 in \cite{ChenLiWang2021}]\label{thm:optr}
  Let \((A,m_n,\mu_{pq})\) be a \(B_{\infty}\)-algebra. Then there is a natural \(B_{\infty}\)-isomorphism between \(A^{\opp}\) and \(A^{\tr}\). 
\end{thm}

\subsection{Operads and \(B_{\infty}\)-algebras}

We recall the definition of (non-symmetric) operads via partial composition and their connection to \(B_{\infty}\)-algebras. We mainly follow the exposition in \cite{GerstenhaberVoronov1995}. Let \((\mathbf{C}, \otimes, \mathbbm{1})\) be a symmetric monoidal category.

\begin{mydef}
  An \textit{(non-symmetric) operad} over \(\mathbf{C}\) is a triple \((\mathcal{O}, \circ, \eta)\), where \(\{\mathcal{O}(n)\}_{n \geq 1}\) is a sequence of objects in \(\mathbf{C}\), where \(\eta : \mathbbm{1} \rightarrow \mathcal{O}(1)\) is a map in \(\mathbf{C}\) and morphisms 
  \begin{equation*}
    \circ_{i}^{m,n} : \mathcal{O}(m) \otimes \mathcal{O}(n) \rightarrow \mathcal{O}(m+n-1) \quad \text{ for } i \in \Z,
  \end{equation*}
  such that the following conditions hold for any \(f \in \mathcal{O}(m)\), \(g \in \mathcal{O}(n)\), \(h \in \mathcal{O}(l)\) (suppressing the superscripts \(m, n\) and \(l\)),
  \begin{enumerate}
  \item \(f \circ_{i} g = 0\) if \(i > m-1\) or \(i < 0\); 
  \item \((f \circ_{i} g) \circ_j h = f \circ_i (g \circ_{j-i} h)\) if \(i \leq j < n+i\); 
  \item \((f \circ_i g) \circ_j h = (f \circ_j h) \circ_{i+l-1} g\) if \(j < i\);
  \item The composition
    \begin{equation*}
      \begin{tikzcd}
        \mathbbm{1} \otimes \mathcal{O}(m) \ar{r}{\eta \otimes \ide} & \mathcal{O}(1) \otimes \mathcal{O}(m) \ar{r}{\circ_{0}} & \mathcal{O}(m) 
      \end{tikzcd}
    \end{equation*}
    is the left unit morphism and the composition
    \begin{equation*}
      \begin{tikzcd}
        \mathcal{O}(m) \otimes \mathbbm{1} \ar{r}{\ide \otimes \eta} & \mathcal{O}(m) \otimes \mathcal{O}(1) \ar{r}{\circ_i} & \mathcal{O}(m)
      \end{tikzcd}
    \end{equation*}
    is the right unit morphism for all \(i < m\).  
  \end{enumerate}
  We call the morphisms \(\circ_i\) \textit{partial compositions}. 
\end{mydef}

If we visualise each element in \(\mathcal{O}(n)\) as a tree
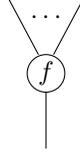
\begin{center}
  \begin{tikzpicture}
    \node at (0,1.75) {\(\dots\)} ;
    
    \draw (0,1) node {\(f\)} circle (0.25);
    
    \draw (0,0) -- (0,0.75);
    \draw (-0.1,1.25) -- (-0.5,2);
    \draw (0.1,1.25) -- (0.5,2);
  \end{tikzpicture}
\end{center}
then the partial composition can be visualised as
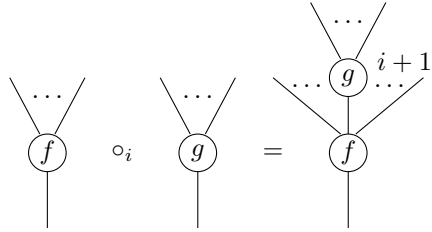
\begin{center}
  \begin{tikzpicture}
    \node at (0,1.75) {\(\dots\)};
    
    \draw (0,1) node {\(f\)} circle (0.25);
    
    \draw (0,0) -- (0,0.75);
    \draw (-0.1,1.25) -- (-0.5,2);
    \draw (0.1,1.25) -- (0.5,2);

    \node at (2,1.75) {\(\dots\)} ;
    
    \draw (2,1) node {\(g\)} circle (0.25);
    
    \draw (2,0) -- (2,0.75);
    \draw (-0.1+2,1.25) -- (-0.5+2,2);
    \draw (0.1+2,1.25) -- (0.5+2,2);

    \node at (1,1) {\(\circ_i\)};
    \node at (3,1) {\(=\)};

    \node at (3.45,1.9) {\(\dots\)};
    \node at (4.55, 1.9) {\(\dots\)};
    \node at (4,2.75) {\(\dots\)};
    \node at (4.75,2.2) {\(i+1\)};
    
    \draw (4,1) node {\(f\)} circle (0.25);
    \draw (4,2) node {\(g\)} circle (0.25);
    
    \draw (4,0) -- (4,0.75);
    \draw (-0.1+4,1.25) -- (-1+4,2);
    \draw (0.1+4,1.25) -- (1+4,2);
    \draw (4,1.25) -- (4,1.75);

    \draw (-0.1+4,1.25+1) -- (-0.5+4,2+1);
    \draw (0.1+4,1.25+1) -- (0.5+4,2+1);
  \end{tikzpicture}
\end{center}
We define the total composition
\begin{equation*}
  \gamma(i_1,...,i_n) : \mathcal{O}(n) \otimes \mathcal{O}(i_1) \otimes ... \otimes \mathcal{O}(i_n) \rightarrow \mathcal{O}(i_1+...+i_n)
\end{equation*}
by \(f \otimes g_{i_1} \otimes ... \otimes g_{i_n} \mapsto (((f \circ_0 g_{i_1}) \circ_{1} ...) \circ_{n-1} g_{n})\). Visually, we graft \(g_{i_1}\) onto the first input of \(f\) and \(g_{i_2}\) onto the second and so on. We denote the image of \(f \otimes g_{i_1} \otimes ... \otimes g_{i_n}\) under \(\gamma\) by \(\gamma(f;g_{i_1},...,g_{i_n})\). A \textit{morphism} \(\varphi : \mathcal{O} \rightarrow \mathcal{O}'\) of operads is a collection of maps \(\{\varphi(n) : \mathcal{O}(n) \rightarrow \mathcal{O}'(n)\}_{n \geq 0}\) in \(\mathbf{C}\) such that they commute with the partial compositions and the unit morphisms. 

\begin{remark}
  Given an operad \(\{\mathcal{O}(n)\}_{n \geq 0}\) over \(\mathbf{Vec}_k\) and an element \(g \in \mathcal{O}(n)\), we let \(|g| = n-1\) denote the degree of \(g\) in the operadic desuspension of \(\mathcal{O}\). We use the same notation for non-negatively graded vector spaces, i.e. if \(f \in V_n\), then \(|f| = n-1\). 
\end{remark}

\begin{ex}
  Let \(\mathbf{C} = \mathbf{Vec}_k\). The \textit{associative operad} \(\mathcal{As}\) is the operad given by \(\mathcal{As}(n) = k\) for all \(n \geq 1\), with partial composition maps given by the canonical isomorphism \(k \otimes_k k \rightarrow k\). 
\end{ex}

\begin{mydef}
  A \textit{brace algebra} \((V,-\{-,...,-\})\) is a non-negatively graded vector space \(V = \oplus_{n \geq 0} V_n\) together with homogeneous maps
  \begin{equation*}
    - \{-,...,-\} : V^{\otimes (n+1)} \rightarrow V
  \end{equation*}
  of degree \(-n\) fulfilling the following equation
  \begin{align*}
    x \{x_1,...,x_n\}\{y_1,...,y_n\} = \sum_{0 \leq i_1 \leq j_1 \leq i_2 \leq ... \leq i_m \leq j_m \leq n}(-1)^{\epsilon}x \{&y_1,...,y_{i_1}, x \{y_{i_1+1},...,y_{j_1}\},y_{j_1+1},..., \\
                                                                                                    &y_{i_m},x_m \{y_{i_m+1},...,y_{j_m}\},y_{j_m+1},...,y_n\}
  \end{align*}
  where \(\epsilon = \sum_{p=1}^{m}|x_p|\sum_{q=1}^{i_p}|y_q|\). We define \(x\{\}\) to be \(x\). Note that, by the equation above,
  \begin{equation*}
    [x,y] = x\{y\} - (-1)^{|x||y|}y\{x\}
  \end{equation*}
  defines a graded Lie algebra structure on \(sV\).
\end{mydef}

Given an operad \(\{\mathcal{O}(n)\}_{n \geq 0}\) over \(\mathbf{Vec}_k\), we define a brace algebra structure on \(\mathcal{O} = \oplus_{n \geq 0} \mathcal{O}(n)\) by
\begin{equation*}
  f \{g_1,...,g_n\} = \sum (-1)^\epsilon \gamma(f;\ide,...,\ide,g_{1},\ide,...,\ide,g_n,\ide,...,\ide)
\end{equation*}
where the sum runs over all possible ways of grafting \(g_1,...,g_n\) onto \(f\) while keeping the order and \(\epsilon = \sum_{p=1}^{n}|g_p|i_p\), with \(i_p\) being the total number of inputs to the left of \(g_p\). The sign comes from the Koszul sign rule applied to expressions of the form \(f \{g_1,...,g_m\}(v_1,...,v_k) = \sum \gamma(f;v_1,...,v_p,g_1(v_{p+1},...,v_l),...,v_k)\).

\begin{ex}
  Example of the signs for the brace operation.
  \begin{center}
    \begin{tikzpicture}
      \draw (0,1) node {\(f\)} circle (0.25);
      \draw (1,1) node {\(g_1\)} circle (0.25);
      \draw (2,1) node {\(g_2\)} circle (0.25);

      \draw ($(-90:0.25)+(0,1)$) -- ($(0:0)+(0,0)$);
      \draw ($(-90:0.25)+(1,1)$) -- ($(0:0)+(1,0)$);
      \draw ($(-90:0.25)+(2,1)$) -- ($(0:0)+(2,0)$);

      \draw ($(110:0.25)+(0,1)$) -- ($(-130:0.25)+(0,2.25)$);
      \draw ($(70:0.25)+(0,1)$) -- ($(-50:0.25)+(0,2.25)$);

      \draw ($(110:0.25)+(1,1)$) -- ($(-130:0.25)+(1,2.25)$);
      \draw ($(70:0.25)+(1,1)$) -- ($(-50:0.25)+(1,2.25)$);

      \draw ($(110:0.25)+(2,1)$) -- ($(-140:0.25)+(2,2.25)$);
      \draw ($(70:0.25)+(2,1)$) -- ($(-40:0.25)+(2,2.25)$);
      \draw ($(90:0.25)+(2,1)$) -- ($(-90:0.25)+(2,2.35)$);
      
      \draw[decoration={brace,raise=5pt},decorate,thick]
      (0.75,0) -- (0.75,2.25);
      \node at (1.5,0) {\(,\)};
      \draw[decoration={brace, mirror, raise=5pt},decorate,thick]
      (2.25,0) -- (2.25,2.25);

      \node at (3,1) {\(=\)};
      \node at (4.25,1) {\((-1)^{0 \cdot 1 + 2 \cdot 2}\)};

      \draw (5.75,1) node {\(f\)} circle (0.25);
      \draw (5,1.75) node {\(g_1\)} circle (0.25);
      \draw (5.75,1.75) node {\(g_2\)} circle (0.25);

      \draw ($(145:0.25)+(5.75,1)$) to[out=150, in=-50, looseness=1] ($(-70:0.25)+(5,1.75)$);
      \draw ($(90:0.25)+(5.75,1)$) -- ($(5.75,1.75)+(-90:0.25)$);
      \draw ($(35:0.25)+(5.75,1)$) to[out=20, in=-90, looseness=1] ($(0:0)+(6.5,2.25)$);
      \draw ($(5.75,0)+(0:0)$) -- ($(-90:0.25)+(5.75,1)$);

      \draw ($(110:0.25)+(5.75,1.75)$) -- ($(5.55,2.25)$);
      \draw ($(70:0.25)+(5.75,1.75)$) -- ($(5.95,2.25)$);
      \draw ($(90:0.25)+(5.75,1.75)$) -- ($(5.75,2.25)$);
      \draw ($(110:0.25)+(5,1.75)$) -- ($(4.8,2.25)$);
      \draw ($(70:0.25)+(5,1.75)$) -- ($(5.2,2.25)$);

      \node at (6.75,1) {\(+\)};
      \node at (8,1) {\((-1)^{0 \cdot 1 + 3 \cdot 2}\)};

      \draw (9.75,1) node {\(f\)} circle (0.25);
      \draw (9,1.75) node {\(g_1\)} circle (0.25);
      \draw (10.5,1.75) node {\(g_2\)} circle (0.25);

      \draw ($(110:0.25)+(9.75+0.75,1.75)$) -- ($(9.55+0.75,2.25)$);
      \draw ($(70:0.25)+(9.75+0.75,1.75)$) -- ($(9.95+0.75,2.25)$);
      \draw ($(90:0.25)+(9.75+0.75,1.75)$) -- ($(9.75+0.75,2.25)$);
      \draw ($(110:0.25)+(9,1.75)$) -- ($(8.8,2.25)$);
      \draw ($(70:0.25)+(9,1.75)$) -- ($(9.2,2.25)$);
      \draw ($(90:0.25)+(9.75,1)$) -- (9.75,2.25);
      \draw ($(-90:0.25)+(9.75,1)$) -- (9.75,0);

      \draw ($(145:0.25)+(9.75,1)$) to[out=150, in=-50, looseness=1] ($(-70:0.25)+(9,1.75)$);
      \draw ($(35:0.25)+(9.75,1)$) to[out=20, in=-130, looseness=1] ($(-110:0.25)+(10.5,1.75)$);

      \node at (11-0.25-7.7,1-2.5) {\(+\)};
      \node at (12-7.75,1-2.5) {\((-1)^{1 \cdot 1 + 3 \cdot 2}\)};

      \draw (13.5-7.75,1-2.5) node {\(f\)} circle (0.25);
      \draw (13.5-7.75,1.75-2.5) node {\(g_1\)} circle (0.25);
      \draw (14.25-7.75,1.75-2.5) node {\(g_2\)} circle (0.25);

      \draw ($(110:0.25)+(13.75+0.75-0.25,1.75)-(7.75,2.5)$) -- ($(13.55+0.75-0.25,2.25)-(7.75,2.5)$);
      \draw ($(70:0.25)+(13.75+0.75-0.25,1.75)-(7.75,2.5)$) -- ($(13.95+0.75-0.25,2.25)-(7.75,2.5)$);
      \draw ($(90:0.25)+(13.75+0.75-0.25,1.75)-(7.75,2.5)$) -- ($(13.75+0.75-0.25,2.25)-(7.75,2.5)$);
      \draw ($(110:0.25)+(13.75-0.25,1.75)-(7.75,2.5)$) -- ($(12.8-0.25+0.75,2.25)-(7.75,2.5)$);
      \draw ($(70:0.25)+(13.75-0.25,1.75)-(7.75,2.5)$) -- ($(13.2-0.25+0.75,2.25)-(7.75,2.5)$);
      \draw ($(-90:0.25)+(13.75-0.25,1)-(7.75,2.5)$) -- ($(13.75-0.25,0)-(7.75,2.5)$);
      \draw ($(90:0.25)+(13.75-0.25,1)-(7.75,2.5)$) -- ($(-90:0.25)+(13.5,1.75)-(7.75,2.5)$);

      \draw ($(35:0.25)+(13.5,1)-(7.75,2.5)$) to[out=20, in=-130, looseness=1] ($(-110:0.25)+(14.25,1.75)-(7.75,2.5)$);
      \draw ($(145:0.25)+(13.5,1)-(7.75,2.5)$) to[out=150, in=-90, looseness=1] ($(12.75,2.25)-(7.75,2.5)$);
    \end{tikzpicture}
  \end{center}
\end{ex}

\begin{prop}
  For any operad \(\{\mathcal{O}(n)\}_{n \geq 0}\) over \(\mathbf{Vec}_k\), the structure described above defines a brace algebra structure on \(\mathcal{O} = \oplus_{n \geq 0}\mathcal{O}(n)\).
\end{prop}

See Proposition 1 in \cite{GerstenhaberVoronov1995}.

\begin{mydef}
  An \textit{operad with multiplication} is an operad \(\mathcal{O}\) over \(\mathbf{Vec}_k\), together with an element \(m \in \mathcal{O}(2)\) such that \(m\{m\} = 0\). A morphism of operads with multiplication \(\varphi : (\mathcal{O},m) \rightarrow (\mathcal{O}',m')\) is morphism of operads such that the following condition holds \(\varphi(2)(m) = m'\). 
\end{mydef}

\begin{remark}
  An operad with multiplication is equivalent to an operad \(\mathcal{O}\) together with the data of a map \(\mathcal{As} \rightarrow \mathcal{O}\) from the associative operad. 
\end{remark}

\begin{ex}
  Given an object \(V\) in a \(k\)-linear monoidal category \((\mathbf{D}, \otimes, \mathbbm{1})\) there are two operads associated to \(V\):
\begin{enumerate}
\item The \textit{endomorphism operad} \(\EndOp(V)\) over \(\mathbf{Vec}_k\), where
  \begin{equation*}
    \EndOp(V)(n) = \mathbf{D}(V^{\otimes n}, V)
  \end{equation*}
  and the partial compositions are given by
  \begin{equation*}
    f \circ_i g = f \circ (\ide^{\otimes i} \otimes g \otimes \ide^{\otimes m-i-1}). 
  \end{equation*}
\item The \textit{coendomorphism operad} \(\CoEndOp(V)\) over \(\mathbf{Vec}_k\), where
  \begin{equation*}
    \CoEndOp(V) = \mathbf{D}(V,V^{\otimes n})
  \end{equation*}
  and the partial compositions are given by
  \begin{equation*}
    f \circ_i g  = (\ide^{\otimes m-i-1} \otimes g \otimes \ide^{\otimes i}) \circ f.
  \end{equation*}
\end{enumerate}
Both of these are operads over \((\mathbf{Vec}_k,\otimes_k)\). If \((V,\mu,\eta)\) is a monoid in \(\mathbf{D}\), then \(\EndOp(V)\) will be an operad with multiplication \(\mu\), whilst if \((V,\Delta,\varepsilon)\) is a comonoid then \(\CoEndOp(V)\) will be an operad with multiplication given by the comultiplication \(\Delta\). Note that \((\mathbf{D}, \otimes, \mathbbm{1})\) does \textbf{not} need to be symmetric, as the (co)endomorphism operads will be operads in \(\mathbf{Vec}_k\) because of the enrichment. 
\end{ex}

Given an operad \(\mathcal{O}\) with multiplication \(\mu\) over \(\mathbf{Vec}_k\), we can define a differential and product on \(\mathcal{O} = \bigoplus_{n \geq 0} \mathcal{O}(n)\) by \(d = [\mu,-]\) and \(f \smile g = (-1)^{}\).
 
\begin{thm}\label{thm:OperadB}
  Let \(\mathcal{O}\) be an operad with multiplication over \(k\). Then \(\bigoplus_{n \geq 0} \mathcal{O}(n)\) has the structure of brace \(B_{\infty}\)-algebra. 
\end{thm}

For a proof, see Theorem 3 in \cite{GerstenhaberVoronov1995}. Let \(\mathbf{mOp}_{k}\) denote the category of operads with multiplication and let \(\mathbf{B}_{\infty, \text{strict}}\) denote the category of \(B_{\infty}\)-algebras with strict morphism. It follows from Lemma~\ref{lemma:strictMorphism} and Theorem~\ref{thm:OperadB}, that we have a functor \(\mathbf{mOp}_k \rightarrow \mathbf{B}_{\infty,\text{strict}}\). 

\subsection{Relative Hochschild cohomology}
\label{sec:relat-hochsch-cohom}

We recall the definition of relative Hoch\-schild cohomology as first defined by Hochschild in 1956 in \cite{Hochschild1956}. Let \(k\) be a field, let \(A\) be a \(k\)-algebra and let \(B \subseteq A\) be a \(k\)-subalgebra. We say that a complex \((X_{\sbu}, d)\) of \(A\)-modules is \((A|B)\)-\textit{relative exact} if it is split exact as a complex of \(B\)-modules. Similarly, a \textit{relative split epimorphism} is an epimorphism \(f : M \rightarrow N\) of \(A\)-modules that splits over \(B\).

\begin{mydef}
  An \(A\)-module \(P\) is \textit{relative projective} if for every relative split epimorphism \(g : Y \rightarrow Z\) and every \(A\)-module map \(\varphi : P \rightarrow Z\) there exists an \(A\)-module map \(\psi : P \rightarrow Y\) such that \(\varphi = g \circ \psi\). 
\end{mydef}

\begin{lemma}
  Let \(P\) be a relative projective \(A\)-module. Then \(P\) is summand of an \(A\)-module \(A \otimes_B N\), where \(N\) is a \(B\)-module.
\end{lemma}

For a proof see Lemma 2 in \cite{Hochschild1956}.

\hfill

A \textit{relative projective resolution} of an \(A\)-module \(M\) is a complex \((P_{n},d)_{n \geq 0}\) such that 
\begin{equation*}
  \begin{tikzcd}
    \dots \ar{r}{} & P_2 \ar{r}{d_2} & P_1 \ar{r}{d_1} & P_0 \ar{r}{\varepsilon} & M \ar{r}{} & 0
  \end{tikzcd}
\end{equation*}
is a relative exact complex and such that \(P_n\) is relative projective for all \(n \geq 0\). Note that every \(A\)-module \(M\) has a relative projective resolution by taking the kernel of the relative split epimorphism \(A \otimes_B M \rightarrow M\). Akin to the absolute case, we define the \textit{relative Ext functor} to be \(\Ext^n_{(A|B)}(M,N) = H^n(\Hom_{A}(P_{\sbu},N))\), where \(P_{\sbu}\) is a relative projective resolution of \(M\). This is independent of the choice of relative projective resolution, see Section 2 in \cite{Hochschild1956}.

\hfill

Let \(M\) be an \(A\)-bimodule. We define the \textit{relative Hochschild cohomology of} \(A\) \textit{with coefficients in} \(M\) to be \(\HH^*(A|B,M) = \Ext^*_{(A^e|B^e)}(A,M)\). If \(M=A\) then we talk simply about the relative Hochschild cohomology of \(A\) and denote it by \(\HH^*(A|B)\). One can compute the relative Hochschild cohomology using the \textit{relative bar resolution} 
\begin{equation*}
  \begin{tikzcd}
    ... \ar{r}{} & A^{\otimes_{B} m} \ar{r}{d} & A^{\otimes_{B} (m-1)} \ar{r} & ... \ar{r}{} & A \otimes_{B} A \ar{r}{\varepsilon} & A \ar{r}{} & 0
  \end{tikzcd}
\end{equation*}
where \(\varepsilon : a \otimes_{B} a' \mapsto aa'\) and the differential \(d : A^{\otimes_{B} (m+1)} \rightarrow A^{\otimes_{B} m}\) is given by
\begin{align*}
  d(a_0 \otimes_{B} ... \otimes_{B} a_{m}) = \sum_{i = 0}^{m-1} (-1)^{i} a_0 \otimes_{B} ... \otimes_{B} a_ia_{i+1} \otimes_{B} ... \otimes_{B} a_m.
\end{align*}
Thus, by applying \(\Hom_{A\text{-}A}(-,M)\) to the relative bar resolution, one gets the complex
\begin{equation*}
  \Hom_{A\text{-}A}(A^{\otimes_B (m+2)}, M) \cong \Hom_{B\text{-}B}(A^{\otimes_B m}, M) 
\end{equation*}
with differential \(d^m : \Hom_{B\text{-}B}(A^{\otimes_B m}, M) \rightarrow \Hom_{B\text{-}B}(A^{\otimes_B (m+1)},M)\) given by
\begin{align*}
  d^m(f)(a_0 \otimes_B ... \otimes_B a_m) &= a_0f(a_1 \otimes_{B} ... \otimes_{B} a_m) \\
                                          &+ \sum_{i=1}^{m-1} (-1)^{i} f(a_0 \otimes_{B} ... \otimes_{B} a_ia_{i+1} \otimes_{B} ... \otimes_{B} a_m) \\
                                          &+ (-1)^{m}f(a_0 \otimes_{B}... \otimes_{B} a_{m-1})a_m
\end{align*}
Let \(C^n(A|B,M)\) denote the relative Hochschild \(n\)-cochains and let \(C^n(A|B)\) denote \(C^n(A|B,A)\). It is noted without proof in \cite{GerstenhaberSchack1986} that relative Hochschild cohomology has the structure of a Gerstenhaber algebra.

\begin{prop}
  The relative Hochschild cochain complex \(C^*(A|B)\) is equal to \(\EndOp(A)\), where \((A,\mu, \iota)\) is seen as a monoid in the category of \(B\)-bimodules. Moreover, it has a multiplication given by \(\mu \in \EndOp(A)(2)\). Therefore, \(C^*(A|B)\) has the structure of a \(B_{\infty}\)-algebra over \(k\) and \(\HH^*(A|B)\) is a Gerstenhaber algebra. 
\end{prop}

\begin{proof}
  This follows from Theorem~\ref{thm:OperadB} and Proposition~\ref{prop:BgivesG}.
\end{proof}

\subsection{Tensor product of relative Hochschild cohomology}
We show that relative Hochschild cohomology commutes with tensor product as Gerstenhaber algebras, following the approach in \cite{OcalOkeWitherspoon2022}. Throughout, let \(B \rightarrow A\) and \(D \rightarrow C\) be extensions of \(k\)-algebras.

\hfill

Note that relative Hochschild cohomology can also be described as Ext-functor in the category of \(A\)-bimodules with the exact structure given by all kernel-cokernel pairs which split as \(B\)-bimodules. Moreover, the relative bar resolution is power flat, i.e \(\mathbb{B}(A|B)^{\otimes_A n}\) is relative exact for all \(n \geq 1\). Therefore, we can use the description of the Gerstenhaber bracket given in \cite{VolkovWitherspoon}. 

\begin{lemma}
  Let \(P_{\sbu} \rightarrow X\) be an \((A|B)\)-relative projective resolution of an \(A\)-bimodule \(X\) and let \(Q_{\sbu} \rightarrow Y\) be a \((C|D)\)-relative projective resolution of a \(C\)-bimodule \(Y\). Then \(P_{\sbu} \otimes Q_{\sbu} \rightarrow X \otimes Y\) is an \((A \otimes C|B \otimes D)\)-relative projective resolution of \(X \otimes Y\) as an \(A \otimes C\)-bimodule.  
\end{lemma}

\begin{proof}
  As \(P_i\) and \(Q_j\) are relative projective, we can assume without loss of generality that \(P_i = A \otimes_B M \otimes_B A\) and \(Q_j = C \otimes_{D} N \otimes_{D} C\) for some \(B\)-bimodule \(M\) and some \(D\)-bimodule \(N\). Thus, we have that \(P_i \otimes Q_j \cong A \otimes C \otimes_{B \otimes D} M \otimes N \otimes_{B \otimes D} A \otimes C\) and thus \(P_i \otimes Q_j\) is \((A \otimes C|B \otimes D)\)-projective. From Künneth's theorem, we get that \(P_{\sbu} \otimes Q_{\sbu}\) is concentrated in degree \(0\) with homology \(H^0(P_{\sbu} \otimes Q_{\sbu}) \cong H^0(P_{\sbu}) \otimes H^0(Q_{\sbu}) \cong A \otimes C\). Lastly, we note, as the tensor product is an additive functor, that \(P_{\sbu} \otimes Q_{\sbu}\) splits as a complex of \(B \otimes D\)-bimodules.
\end{proof}

To show that we have an isomorphism of Gerstenhaber algebras, we use the description of the Gerstenhaber bracket given in \cite{Volkov2019}. Let \(P_{\sbu} \rightarrow A\) be a relative projective resolution of \(A\) as a bimodule and let \(Q_{\sbu} \rightarrow C\) be a relative projective resolution of \(C\) as a bimodule. Let \(f : P \rightarrow A[-m]\) and \(g : Q \rightarrow C[-n]\) be two cocycles, and let \(\varphi_f : P \rightarrow P[1-m]\) and \(\varphi_g : Q \rightarrow C[1-n]\) be their respective lifts. We have the map
\begin{equation*}
  (f \otimes 1)\Delta_P : P \rightarrow P[-m] 
\end{equation*}
given by \(p \mapsto f(p_{\underline{1}})p_{\underline{2}}\) and the map
\begin{equation*}
  (1 \otimes g)\Delta_Q : Q \rightarrow Q[-n]
\end{equation*}
given by \(q \mapsto (-1)^{|q_{\underline{1}}|\cdot n}q_{\underline{1}}g(q_{\underline{2}})\). Thus, we get the maps
\begin{equation*}
  \varphi_f \otimes (1 \otimes g)\Delta_Q : P \otimes Q \rightarrow (P \otimes Q)[1-m-n] 
\end{equation*}
given by \(p \otimes q \mapsto (-1)^{((|\varphi_f(p)|+1-m)+|q_{\underline{1}}|)(-n)}\varphi_f(q) \otimes q_{\underline{1}}g(q_{\underline{2}})\) and
\begin{equation*}
  (f \otimes 1)\Delta_P \otimes \varphi_g : P \otimes Q \rightarrow (P \otimes Q)[1-m-n]
\end{equation*}
given by \(p \otimes q \mapsto (-1)^{(|f(p_{\underline{1}})p_{\underline{2}}|-m)(-n)}f(p_{\underline{1}})p_{\underline{2}} \otimes \varphi_g(q)\).

\begin{lemma}
  Let \(f\) and \(g\) be two cocycles for \(P_{\sbu}\) and \(Q_{\sbu}\) respectively and let \(\varphi_f\) and \(\varphi_g\) be their respective lifts. Then
  \begin{equation*}
    \varphi_{f \otimes g} = \varphi_f \otimes (1 \otimes g)\Delta_C + (-1)^m(f \otimes 1)\Delta_A \otimes \varphi_g
  \end{equation*}
  is a lift for \(f \otimes g\). 
\end{lemma}

For a proof, see Lemma 3.8 in \cite{OcalOkeWitherspoon2022} for the absolute case. 

\begin{thm}\label{thm:tensorDecomp}
  The isomorphism
  \begin{equation*}
    \begin{tikzcd}
      \HH^*(A | B) \otimes \HH^*(C | D) \ar{r}{} & \HH^*(A \otimes C | B \otimes D)
    \end{tikzcd}
  \end{equation*}
  given by \([f] \otimes [g] \mapsto [f \otimes g]\) is an isomorphism of Gerstenhaber algebras. 
\end{thm}

The proof is more or less identical to the proof in the absolute case, see Theorem 3.9 in \cite{OcalOkeWitherspoon2022}. 

\subsection{Corings and bicomodules}
\label{sec:corings}

Let \(B\) be a ring. We denote morphisms of left, respectively right, modules by \(\Hom_{B\text{-}}(-,-)\) and \(\Hom_{\text{-}B}(-,-)\) respectively.

\begin{mydef}
  Let \(P\) be a left \(B\)-module. A \textit{left finite dual basis} for \(P\) is a set \(\{f_i,a_i\}_{i=1,...,n}\), where \(f_i \in \Hom_{B\text{-}}(P,B)\) and \(a_i \in P\), such that \(x = \sum_{i=1}^n f_i(x)a_i\) for all \(x \in P\). To reduce clutter, we will skip the summation sign and write \(x = f_i(x)a_i\). We define a finite dual basis of right modules and bimodules similarly. 
\end{mydef}

\begin{lemma}
  Let \(P\) be a left \(B\)-module. Then \(P\) is finitely generated projective if and only if there exists a finite dual basis \(\{f_i,a_i\}_{i=1,...,n}\) for \(P\). 
\end{lemma}

For a proof see Remark 2.11 in \cite{Lam1999}.  

\begin{mydef}
  Let \(B\) be a ring. A \(B\)\textit{-coring} \(\mathcal{C}\) is a comonoid object in the monoidal category of \(B\)-bimodules. Spelled out, \(\mathcal{C}\) is a \(B\)-bimodule together with \(B\)-bimodule maps \(\Delta : \mathcal{C} \rightarrow \mathcal{C} \otimes_{B} \mathcal{C}\) and \(\varepsilon : \mathcal{C} \rightarrow B\) such that the following diagrams commute:
  \begin{equation*}
    \begin{tikzcd}
      \Co \ar{r}{\Delta} \ar[swap]{d}{\Delta} & \Co \otimes_A \Co \ar{d}{\Delta \otimes 1} && \Co \ar{r}{\Delta} \ar{dr}{1} \ar[swap]{d}{\Delta} & \Co \otimes_A \Co \ar{d}{\varepsilon \otimes 1}\\
    \Co \otimes_A \Co \ar[swap]{r}{1 \otimes \Delta} & \Co \otimes_A \Co \otimes_A \Co && \Co \otimes_A \Co \ar[swap]{r}{1 \otimes \varepsilon} & \Co
    \end{tikzcd}
  \end{equation*}
  We call the pair \((B,\mathcal{C})\) a bocs to be consistent with both the literature on representation theory and Hopf algebras.
\end{mydef}

\begin{lemma}\label{Lemma:leftdecomp}
  Let \(M,N\) be \(B\)-bimodules. Then we have a morphism of \(B\)-bimodules
  \begin{equation*}
    \begin{tikzcd}
      \gamma : \Hom_{B\text{-}}(N,B) \otimes_{B} \Hom_{B\text{-}}(M,B) \ar{r} & \Hom_{B\text{-}}(M \otimes_{B} N,B)
    \end{tikzcd}
  \end{equation*}
  given by \(\sigma \mapsto [m \otimes_{B} n \mapsto g(mf(n))]\). Moreover, if \(N\) is left finitely generated projective, then \(\gamma\) is an isomorphism with inverse given by \(\sigma \mapsto f_i \otimes_B \sigma(- \otimes_{B} a_i)\), where \(\{a_i,f_i\}\) is dual basis for \(N\). 
\end{lemma}

For a proof, see Lemma 3.4 in \cite{Sweedler1975}. 

\begin{lemma}\label{lemma:fdb}
  Let \(\{m_i,f_i\}\) be a left finite dual basis for a left finitely generated projective \(B\)-bimodule \(M\) and let \(\{n_i, g_i\}\) be a left finite dual basis for a left finitely generated projective \(B\)-bimodule \(N\). Then
  \begin{enumerate}
  \item \(\{f_i, \ev_{m_i}\}\) is a right finite dual basis for \(\Hom_{B\text{-}}(M,B)\);
  \item \(\{m_i \otimes_{B} n_j, h_{ij} = [m \otimes_{B} n \mapsto f_i(mg_j(n))]\}\) is a left finite dual basis for \(M \otimes_{B} N\).
  \end{enumerate}
\end{lemma}

Part (1) is shown in Lemma 3.5 in \cite{Sweedler1975}. We give a proof of both parts here for completeness. 

\begin{proof}\hfill
  \begin{enumerate}
  \item Let \(\varphi \in \Hom_{B\text{-}}(M,B)\) then we have that
    \begin{align*}
      (f_i \cdot \ev_{m_i}(\varphi))(x) &= (f_i \cdot \varphi(m_i))(x) = f_i(x)\varphi(m_i) = \varphi(f_i(x)m_i) = \varphi(x)
    \end{align*}
    for all \(x \in M\). Therefore, \(f_i \cdot \ev_{m_i}(\varphi) = \varphi\) for all \(\varphi \in \Hom_{B\text{-}}(M,B)\). 
  \item Note that
    \begin{align*}
      h_{ij} \cdot m_i \otimes_{B} n_j &= f_i(mg_j(n)) \cdot m_i \otimes_{B} n_j \\
                                       &= f_i(mg_j(n))m_i \otimes_{B} n_j \\
                                       &= mg_j(n) \otimes_{B} n_j \\
                                       &= m \otimes_{B} g_j(n)n_j = m \otimes_{B} n. 
    \end{align*}
    Therefore \(\{m_i \otimes_{B} n_j, [m \otimes_{B} n \mapsto f_i(mg_j(n))]\}\) is a left finite dual basis for \(M \otimes_B N\).\qedhere
  \end{enumerate}
\end{proof}

Let \(A\) be a ring and let \(B \subseteq A\) be a subring. Then \(\mathcal{C} = \Hom_{B\text{-}}(A,B)\) can be given the structure of a \(B\)-bimodule by \((b_1 \cdot f \cdot b_2)(x) = b_1f(b_2x)\). We define the map \(\varepsilon : \mathcal{C} \rightarrow B\) by \(f \mapsto f(1)\). Further, if \(A\) is finitely generated projective as a right module, then we define \(\Delta : \mathcal{C} \rightarrow \mathcal{C}\otimes_{B}\mathcal{C}\) to be the \(B\)-bimodule map that makes the following diagram commute
\begin{equation*}
  \begin{tikzcd}
    \mathcal{C} \ar[equal]{r}{} \ar[swap]{d}{\Delta} & \Hom_{\text{-}B}(A,B) \ar{d}{(\text{mult})^*} \\
    \mathcal{C} \otimes_{B} \mathcal{C} & \Hom_{\text{-}B}(A \otimes_{B} A,B) \ar[]{l}{\gamma^{-1}}
  \end{tikzcd}
\end{equation*}
where \(\gamma^{-1}\) is the inverse of the isomorphism given in Lemma~\ref{Lemma:leftdecomp}. We call \(\mathcal{C}\) the \textit{right dual coring} associated to \(B \subseteq A\). The dual coring was first defined by Sweedler in \cite{Sweedler1975}. For a proof that this is indeed a coring, see Theorem 3.7 in \cite{Sweedler1975}. 

\hfill

Given a \(B\)-coring \(\mathcal{C}\), we can give the \(B\)-bimodule \(\Hom_{B\text{-}}(\mathcal{C},B)\) an algebra structure by letting \(f \cdot g \in \Hom_{B\text{-}}(\mathcal{C},B)\) be the following composition
\begin{equation*}
  \begin{tikzcd}
    f \cdot g : \mathcal{C} \ar{r}{\Delta} & \mathcal{C} \otimes_{B} \mathcal{C} \ar{r}{1 \otimes_B f} & \mathcal{C} \ar{r}{g} & B.
  \end{tikzcd}
\end{equation*}
Note that \(\varepsilon : \mathcal{C} \rightarrow B\) is a unit for the multiplication. For a proof, see Proposition 3.2 in \cite{Sweedler1975}. The \textit{right Burt-Butler algebra} of the coring \(\mathcal{C}\) or \textit{the opposite of the left dual algebra}, is defined to be \(R := \Hom_{B\text{-}}(\mathcal{C},B)^{\op}\), see Section 1 in \cite{BurtButler1991} or Section 3 in \cite{Sweedler1975} respectively. Note that we have that \(B\) is a subalgebra of \(R\). If \(B \subseteq A\) is an extension of a \(k\)-algebra such that \(A\) is finitely generated projective as a right \(B\)-module, then the right algebra \(R\) of the dual coring is isomorphic to \(A\). See Theorem 3.7 in \cite{Sweedler1975} for a proof. 

\hfill

Let \(\mathcal{C}\) be a \(B\)-coring. A \(\mathcal{C}\)-\textit{bicomodule} is a \(B\)-bimodule together with \(B\)-bimodule maps \(\lambda_M : M \rightarrow \mathcal{C} \otimes_{B} M\) and \(\rho_M : M \rightarrow M \otimes_{B} \mathcal{C}\) such that the following diagrams commute
\[\begin{tikzcd}
	M & {M \otimes_B \mathcal{C}} & M & {M \otimes_B \mathcal{C}} & {\phantom{A}} \\
	{M \otimes_B \mathcal{C}} & {M \otimes_B \mathcal{C} \otimes_B \mathcal{C}} && M & {\phantom{A}} \\
	M & {\mathcal{C} \otimes_B M} & M & {\mathcal{C} \otimes_B M} & {\phantom{A}} \\
	{\mathcal{C} \otimes_B M} & {\mathcal{C} \otimes_B \mathcal{C} \otimes_B M} && M & {\phantom{A}} \\
	M & {M \otimes_B \mathcal{C}} &&& {\phantom{A}} \\
	{\mathcal{C} \otimes_B M} & {\mathcal{C} \otimes_B M \otimes_B \mathcal{C}} &&& {\phantom{A}}
	\arrow["{\rho_M}", from=1-1, to=1-2]
	\arrow["{\rho_M}"', from=1-1, to=2-1]
	\arrow["{1 \otimes_B \Delta}"', from=2-1, to=2-2]
	\arrow["{\rho_M \otimes_B 1}", from=1-2, to=2-2]
	\arrow[from=1-3, to=2-4]
	\arrow["{\rho_M}", from=1-3, to=1-4]
	\arrow["{1 \otimes_B \varepsilon}", from=1-4, to=2-4]
	\arrow["{\lambda_M}", from=3-1, to=3-2]
	\arrow["{\lambda_M}"', from=3-1, to=4-1]
	\arrow["{\Delta \otimes_B 1}"', from=4-1, to=4-2]
	\arrow["{1 \otimes_B \lambda_M}", from=3-2, to=4-2]
	\arrow[from=3-3, to=4-4]
	\arrow["{\varepsilon \otimes_B 1}", from=3-4, to=4-4]
	\arrow["{\lambda_M}", from=3-3, to=3-4]
	\arrow["{\rho_M}", from=5-1, to=5-2]
	\arrow["{\lambda_M}"', from=5-1, to=6-1]
	\arrow["{1 \otimes_B \rho_M}"', from=6-1, to=6-2]
	\arrow["{\lambda_M \otimes_B 1}", from=5-2, to=6-2]
	\arrow["{\text{Right comodule}}", draw=none, from=1-5, to=2-5]
	\arrow["{\text{Left comodule}}", draw=none, from=3-5, to=4-5]
	\arrow["{\text{Compatibility}}", draw=none, from=5-5, to=6-5]
\end{tikzcd}\]

We call the map \(\lambda\) the \textit{left coaction} and \(\rho\) the \textit{right coaction}. Note that \(\mathcal{C}\) is a \(\mathcal{C}\)-bicomodule with the left and right coaction being the comultiplication.

\hfill

Let \(M\) and \(N\) be two \(\mathcal{C}\)-bicomodules and let \(f : M \rightarrow N\) be a morphism of the underlying \(B\)-bimodules. Then \(f : M \rightarrow N\) is a morphism of \(\mathcal{C}\)-bicomodules if the following diagrams commute

\[\begin{tikzcd}[column sep = large]
	M & N & {} & M & N \\
	{\mathcal{C} \otimes_B M} & {\mathcal{C} \otimes_B M} && {M \otimes_B \mathcal{C}} & {N \otimes_B \mathcal{C}}
	\arrow["f", from=1-1, to=1-2]
	\arrow["{\lambda_M}"', from=1-1, to=2-1]
	\arrow["{\lambda_N}", from=1-2, to=2-2]
	\arrow["f", from=1-4, to=1-5]
	\arrow["{\rho_M}"', from=1-4, to=2-4]
	\arrow["{\rho_N}", from=1-5, to=2-5]
	\arrow["{\ide_{\mathcal{C}} \otimes_B f }"', from=2-1, to=2-2]
	\arrow["{f \otimes_B \ide_{\mathcal{C}}}"', from=2-4, to=2-5]
\end{tikzcd}\]

\begin{prop}\hfill
  \begin{enumerate}
  \item The category \(\mathcal{C}\)-bicomodules has cokernels, which coincide with those in \(B\)-bimodules. 
  \item Let \(f : M \rightarrow N\) be a bicomodule map and let \(\ker{f}\) be its kernel as a \(B\)-bimodule map. If the injection \(\iota : \ker{f} \rightarrow M\) splits as a map of \(B\)-bimodules, then \(\ker{f}\) is a subbicomodule of \(M\). Moreover, it is the kernel of \(f\) in the category of bicomodules. 
  \end{enumerate}
\end{prop}

For a proof, see Proposition~1.1 in \cite{Guzman1989}. We denote the set of bicomodule morphism by \(\Hom^{\mathcal{C}-\mathcal{C}}(-,-)\). 

\subsection{Cartier cohomology}
\label{sec:cartier-cohomology}

We recall Cartier cohomology for corings, first defined by Cartier for coalgebras in 1956, see \cite{Cartier1955}, and later extended to corings, see \cite{Guzman1989}. For a more detailed exposition, see Chapter 30 in \cite{CoringsBW}. Let \(B\) be a ring and let \(\mathcal{C}\) be a \(B\)-coring. We say that a complex of \(\mathcal{C}\)-bicomodules is \textit{relative exact} if it is split exact as a complex of \(B\)-bimodules. Similarly, a \textit{relative split monomorphism} is a monomorphism \(f : M \rightarrow N\) of \(\mathcal{C}\)-bicomodules which splits over \(B\).

\begin{mydef}
  A \(\mathcal{C}\)-bicomodule \(I\) is \textit{relative injective} if for every relative split monomorphism \(g : Y \rightarrow Z\) and every \(\mathcal{C}\)-bicomodule map \(\varphi : Y \rightarrow I\) there exists an \(\mathcal{C}\)-bicomodule map \(\psi : Z \rightarrow I\) such that \(\varphi = \psi \circ g\). 
\end{mydef}

\begin{lemma}
  Let \(I\) be a relative injective \(\mathcal{C}\)-bicomodule. Then \(I\) is a summand of a \(\mathcal{C}\)-bicomodule \(\mathcal{C} \otimes_B N \otimes_{B} \mathcal{C}\), where \(N\) is a \(B\)-bimodule.
\end{lemma}

The proof is similar to Lemma 2 in \cite{Hochschild1956}, we will leave it to the reader.

\hfill

A \textit{relative injective resolution} of a \(\mathcal{C}\)-bicomodule \(M\) is a complex \((I^{n},d)_{n \geq 0}\) such that 
\begin{equation*}
  \begin{tikzcd}
    0 \ar{r}{} & M \ar{r}{\varepsilon} & I^0 \ar{r}{d^2} & I^1 \ar{r}{d^1} & I^2 \ar{r}{\varepsilon} & ...
  \end{tikzcd}
\end{equation*}
is a relative exact complex and such that \(I^n\) is relative injective for all \(n \geq 0\). Note that every \(\mathcal{C}\)-bicomodule \(M\) has a relative injective resolution by taking the cokernel of the relative split monomorphism \(M \rightarrow \mathcal{C} \otimes_B M \otimes_{B} \mathcal{C}\) given by the bicomodule coaction. We define the \textit{relative Ext functor} to be \(\Ext^n_{(\mathcal{C}\text{-}\mathcal{C}|B\text{-}B)}(M,N) = H^n(\Hom^{\mathcal{C}}(M,I^{\sbu}))\), where \(I^{\sbu}\) is relative injective resolution of \(M\). This is independent of the choice of relative injective resolution, similarly to Section 2 in \cite{Hochschild1956}.

\hfill

We define the \textit{Cartier cohomology of} \(\mathcal{C}\) with coefficients in a bicomodule \(M\) to be \(\Ext^*_{(\mathcal{C}\text{-}\mathcal{C}|B\text{-}B)}(M,\mathcal{C})\). In the case where \(M=\mathcal{C}\), we speak simply about the \textit{Cartier cohomology of} \(\mathcal{C}\). One can compute the Cartier cohomology by the following resolution. The \textit{cobar resolution} \(\Cob(\mathcal{C})\) is the complex given by
\begin{equation*}
  \Cob(\mathcal{C})^n = \mathcal{C}^{\otimes_{B} (n+2)}
\end{equation*}
with differential
\begin{equation*}
  d^n = \sum_{k=0}^{n+1} (-1)^k 1^{\otimes_{B} k} \otimes_{B} \Delta \otimes_{B} 1^{\otimes_{B} (n-k+1)} : \Cob(\mathcal{C})^n \rightarrow \Cob(\mathcal{C})^{n+1}.
\end{equation*}
We have a coaugmentation map given by \(\Delta : \mathcal{C} \rightarrow \mathcal{C} \otimes_{B} \mathcal{C}\). We can apply \(\Hom^{\mathcal{C}\text{-}\mathcal{C}}(M,-)\) to the cobar resolution to get the complex
\begin{equation*}
  \Hom^{\mathcal{C}\text{-}\mathcal{C}}(M,\mathcal{C}^{\otimes_{B} (n+2)}) \cong \Hom_{B\text{-}B}(M,\mathcal{C}^{\otimes_B n})
\end{equation*}
with differential \(d^n : \Hom_{B\text{-}B}(M,\mathcal{C}^{\otimes_{B} n}) \rightarrow \Hom_{B\text{-}B}(M,\mathcal{C}^{\otimes_{B}(n+1)})\) given by 
\begin{equation*}
  d^n(f) = (1 \otimes_{B} f) \circ \lambda_M + \sum_{k = 1}^{n}(-1)^k(1^{\otimes_{B} k} \otimes_{B} \Delta \otimes_{B} 1^{\otimes_{B} (n-k)}) \circ f + (-1)^{n+1} (f \otimes_{B} 1) \circ \rho_M.
\end{equation*}
By taking the cohomology, we get the Cartier cohomology. Let \(C^n(\mathcal{C}, M)\) denote the Cartier \(n\)-cochains and let \(C^n(\mathcal{C}) := C^n(\mathcal{C},\mathcal{C})\). 

\begin{prop}[Proposition~30.8 in \cite{CoringsBW}]\label{prop:CartierBinfinity}
  The Cartier cochain complex \(C^n(\mathcal{C})\) is equal to \(\CoEndOp(\mathcal{C})\), where \((\mathcal{C}, \Delta, \varepsilon)\) is seen as a comonoid in the category of \(B\)-bimodules. Moreover, it has a multiplication given by \(\Delta \in \CoEndOp(\mathcal{C})(2)\). Therefore, \(C^n(\mathcal{C})\) has the structure of a \(B_{\infty}\)-algebra and \(\Hh_{\Ca}^*(\mathcal{C})\) is a Gerstenhaber algebra.   
\end{prop}

\begin{proof}
  This follows from Theorem~\ref{thm:OperadB} and Proposition~\ref{prop:BgivesG}.
\end{proof}

\subsection{Differential graded Lie algebras and Maurer-Cartan elements}

We recall the definition of differential graded Lie algebras and the basics of deformation theory following Chapter 6 in \cite{BarmeierWang2023}. Throughout, we work over a field \(k\) of characteristic zero. 

\begin{mydef}
  A \textit{differential graded (dg) Lie algebra} is a triple \((\mathfrak{g},d,[-,-])\), where \(\mathfrak{g}\) is a \(\Z\)-graded vector space, \(d : \mathfrak{g} \rightarrow \mathfrak{g}\) is a linear map of degree \(1\) such that \(d^2 = 0\), and \([-,-] : \mathfrak{g} \otimes \mathfrak{g} \rightarrow \mathfrak{g}\) is a map of degree 0 such that the following properties are satisfied:
  \begin{enumerate}
  \item \([x,y] = (-1)^{|x||y|+1}[y,x]\); 
  \item \(d([x,y]) = [d(x),y]+(-1)^{|x|}[x,d(y)]\);
  \item \((-1)^{|x||z|}[x,[y,z]] + (-1)^{|y||x|}[y,[z,x]] + (-1)^{|z||y|}[z,[x,y]] = 0\).
  \end{enumerate}
\end{mydef}

\begin{ex}
  Let \((A,m_n,\mu_{pq})\) be a \(B_{\infty}\)-algebra and let
  \begin{equation*}
    [f,g] = (-1)^{|f|}\mu_{1,1}(f,g) - (-1)^{(|f|-1)(|g|-1)+|g|}\mu_{1,1}(g,f)
  \end{equation*}
  be the associated bracket. Then \((sA,d=m_1,[-,-])\) is a dg Lie algebra. 
\end{ex}

\begin{mydef}
  Let \(\mathfrak{g}\) be a dg Lie algebra. A \textit{Maurer-Cartan element in} \(\mathfrak{g}\) is a element \(x\) of degree \(1\) such that \(d(x) + \frac{1}{2}[x,x] = 0\). We denote the set of Maurer-Cartan elements by \(\mathbf{MC}(\mathfrak{g})\). 
\end{mydef}

\begin{ex}
  Let \(\mathbf{D}\) be a \(k\)-linear monodial category, let \(V\) be an object in \(\mathbf{D}\) and let \((\EndOp(V),-\{-\})\) be the associated endomorphism operad with its brace algebra structure. Then \(\mu \in \EndOp(V)(2) = \mathbf{D}(V \otimes V,V)\) is associative if and only if \(\mu \in \mathbf{MC}(\EndOp(V),[-,-])\). 
\end{ex}

\begin{mydef}
  Let \(\varphi : B \rightarrow A\) be an extension of \(k\)-algebras. A \textit{formal deformation} of \(A\) is a \(k[[t]]\)-linear map
  \begin{equation*}
    \mu_t : A[[t]] \otimes_{k[[t]]} A[[t]] \rightarrow A[[t]]
  \end{equation*}
  such that \(\mu_t = \mu + \mu_1t + \mu_2t^2 + ...\), where \(\mu\) is the multiplication for \(A\) and all \(\mu_i : A \otimes A \rightarrow A\) are linear maps such that \((A[[t]],\mu_t)\) is an algebra over \(k[[t]]\). The deformation is \(B\)\textit{-relative} if \(\mu_t(a,\varphi(b)) = \mu(a,\varphi(b))\) and \(\mu_t(\varphi(b),a) = \mu(\varphi(b),a)\).  
\end{mydef}

By general theory of deformations for dg Lie algebra, we get the following 

\begin{prop}\label{prop:deformationMC}
  There is a one-to-one correspondence
  \begin{equation*}
    \begin{tikzcd}
      \{B\text{-relative deformations}\} \ar[leftrightarrow]{r}{1:1} & \{\mathbf{MC}(sC^*(R|B) \otimes (t))\}
    \end{tikzcd}
  \end{equation*}
  where \(C^*(R|B)\) denotes the relative Hochschild cochains and \((t)\) is the maximal ideal of the power series ring \(k[[t]]\). 
\end{prop}

For a proof of the general case, see Proposition 1.59 in \cite{DotsenkoShadrinVallette2024}.  

\begin{prop}\label{prop:opMC}
  There is a one-to-one correspodence
  \begin{equation*}
    \begin{tikzcd}
      \{B\text{-relative deformations of } R\} \ar[leftrightarrow]{r}{1:1} & \{B^{\op}\text{-relative deformations of } R^{\op}\}
    \end{tikzcd}
  \end{equation*}
  given by sending a deformation \(\phi_t = \phi + \phi_1t + ...\) to the deformation \(\phi_t^{\op} = \phi^{\op} + \phi_1^{\op}t + ...\), which therefore induces a one-to-one correspondence
  \begin{equation*}
    \begin{tikzcd}
      \{\mathbf{MC}(sC^*(R|B) \otimes (t))\} \ar[leftrightarrow]{r}{1:1} & \{\mathbf{MC}(sC^*(R^{\op}|B^{\op}) \otimes (t))\}.
    \end{tikzcd}
  \end{equation*}
\end{prop}

\section{Cartier cohomology and relative Hochschild cohomology}
\label{sec:cart-cohom-relat}

Throughout, let \(k\) be a field, \(B\) a \(k\)-algebra, \(\mathcal{C}\) a \(B\)-coring and \(R\) be the right algebra of \(\mathcal{C}\).

\hfill

We recall the functors described in Chapter 19 in \cite{CoringsBW}. Let \(M\) be a right \(\mathcal{C}\)-comodule with structure map \(\rho_M\). We can define a left \(\Hom_{B\text{-}}(\mathcal{C},B)\)-module structure on \(\Hom_{B\text{-}}(M,B)\) by 
\begin{equation*}
  \begin{tikzcd}
    r \cdot x: M \ar{r}{\rho} & M \otimes \mathcal{C} \ar{r}{1 \otimes r} & M \ar{r}{x} & B.
  \end{tikzcd}
\end{equation*}
for all elements \(r \in \ld{\mathcal{C}}\) and all elements \(x \in \Hom_{B\text{-}}(M,B)\). Note that \(r\cdot x\) is equal to the composition

\begin{equation*}
  \begin{tikzcd}[column sep = large]
    \Hom_{B\text{-}}(\mathcal{C}, B) \otimes_{B} \Hom_{B\text{-}}(M,B) \ar{r}{\gamma} & \Hom_{B\text{-}}(M \otimes_{B} \mathcal{C},B) \ar{r}{\Hom_{B\text{-}}(\rho,B)} & \Hom_{B\text{-}}(M,B)
  \end{tikzcd}
\end{equation*}

where \(\gamma\) is the map defined in Lemma~\ref{Lemma:leftdecomp}. Similarly, if \(M\) is a left \(\mathcal{C}\)-comodule, we get a right \(\Hom_{B\text{-}}(\mathcal{C},B)\)-module structure on \(\Hom_{B\text{-}}(M,B)\) by

\begin{equation*}
  \begin{tikzcd}
    x \cdot r : M \ar{r}{\lambda} & \mathcal{C} \otimes_{B} M \ar{r}{1 \otimes_{B} x} & \mathcal{C} \ar{r}{r} & B. 
  \end{tikzcd}
\end{equation*}

Thus, as \(R = \Hom_{B\text{-}}(\mathcal{C},B)^{\op}\), every left \(\mathcal{C}\)-comodule gives a left \(R\)-module and every right \(\mathcal{C}\)-comodule gives a right \(R\)-module.

\begin{prop}
  We have a faithful functor
  \begin{equation*}
    F_R : \mathcal{C}\Bicomod \rightarrow R\Bimod
  \end{equation*}
  defined by \(M \mapsto \Hom_{B\text{-}}(M,B)\) and \(f \mapsto \Hom_{B\text{-}}(f,B)\). 
\end{prop}

\begin{proof}
  First, we need to show that the left and right action are compatible, i.e that
  \begin{equation*}
    r \cdot (x \cdot t) = (r \cdot x) \cdot t.
  \end{equation*}
  We note that \(r \cdot (x \cdot t)\) is equal to the composition
  \begin{equation*}
    \begin{tikzcd}
      M \ar{r}{\rho} & M \otimes_{B} \mathcal{C} \ar{r}{1 \otimes_{B} r} & M \ar{r}{\lambda} & \mathcal{C} \otimes_{B} M \ar{r}{1 \otimes_{B} x} & \mathcal{C} \ar{r}{t} & B.
    \end{tikzcd}
  \end{equation*}
  Moreover, by coassociativity the composition
  \begin{equation*}
    \begin{tikzcd}
      M \ar{r}{\rho} & M \otimes_{B} \mathcal{C} \ar{r}{1 \otimes_{B} r} & M \ar{r}{\lambda} & \mathcal{C} \otimes_{B} M
    \end{tikzcd}
  \end{equation*}
  is equal to the composition
  \begin{equation*}
    \begin{tikzcd}[column sep = large]
      M \ar{r}{\rho} & M \otimes_{B} \mathcal{C} \ar{r}{\lambda} & \mathcal{C} \otimes_{B} M \otimes_{B} \mathcal{C} \ar{r}{1 \otimes_{B} 1 \otimes_{B} r} & \mathcal{C} \otimes_{B} M. 
    \end{tikzcd}
  \end{equation*}
  Therefore, as \(M\) was assumed to be a \(\mathcal{C}\)-bicomodule, the actions are compatible. Let \(f : M \rightarrow N\) be a morphism of \(\mathcal{C}\)-bicomodules. It follows from Proposition 19.1 in \cite{CoringsBW} that \(F_R(f)\) is a morphism of \(R\)-bimodules and that \(F_R\) is faithful.
\end{proof}

\begin{prop}\label{prop:rightRelExact}
  The functor \(F_R : \mathcal{C}\Bicomod \rightarrow R\Bimod\) sends relative exact complexes to relative exact complexes. 
\end{prop}

\begin{proof}
  Note that the following diagram commutes, where the \(U\)s are the forgetful functors.
  \begin{equation*}
    \begin{tikzcd}[column sep = large]
      \mathcal{C}\Bicomod \ar{r}{F_R} \ar[swap]{d}{U} & R\Bimod \ar{d}{U} \\
      B\Bimod \ar[swap]{r}{\Hom_{B\text{-}}(-,B)} & B\Bimod 
    \end{tikzcd}
  \end{equation*}
  Therefore it suffices to show that \(\Hom_{B\text{-}}(-,B)\) sends relative exact complexes to relative exact complexes, which then implies that \(F_R\) fulfills the property by the commutativity of the diagram. Let \((X_{\sbu},d)\) be a split exact complex of \(B\)-bimodules. As \(\Hom_{B\text{-}}(-,B)\) is an additive functor, we have that \(\Hom_{B\text{-}}(X_{\sbu},B)\) is a split exact complex of \(B\)-bimodules. Therefore, \(\Hom_{B\text{-}}(-,B)\) sends relative exact complexes to relative exact complexes. 
\end{proof}
  
\begin{prop}\label{prop:extduality}
  Let \(M\) and \(N\) be two \(\mathcal{C}\)-comodules. We have a natural map
  \begin{equation*}
    \begin{tikzcd}
      \Ext^*_{(\mathcal{C}\text{-}\mathcal{C}|B\text{-}B)}(M,N) \rightarrow \Ext^*_{(A^e|B^e)}(F_R(N),F_R(M))
    \end{tikzcd}
  \end{equation*}
  of \(k\)-vector spaces. 
\end{prop}

\begin{proof}
  Let \(I^{\sbu}\) be a relative injective resolution of \(N\). By Proposition~\ref{prop:rightRelExact}, \(F_R(I^{\sbu})\) is a relative exact complex. Let \(P_{\sbu}\) be a relative projective resolution of \(F_R(N)\). By lifting the identity morphism, we get a map 
  \begin{equation*}
    \begin{tikzcd}
      P_{\sbu} \ar{r}{} \ar[swap]{d}{\varphi} & F_R(N) \ar{d}{\ide_{F_R(N)}} \\
      F_R(I^{\sbu}) \ar{r}{} & F_R(N)
    \end{tikzcd}
  \end{equation*}
  Therefore we get a map of chain complexes by composition
  \begin{equation*}
    \begin{tikzcd}
      \Hom^{\mathcal{C}\text{-}\mathcal{C}}(M,I^{\sbu}) \ar{r}{(F_R)_{M,I^{\sbu}}} & \Hom_{A\text{-}A}(F_R(I^{\sbu}),F_R(M)) \ar{r}{} & \Hom_{A\text{-}A}(P_{\sbu},F_R(M))
    \end{tikzcd}
  \end{equation*}
  Note that the map is natural up to homotopy by construction. Thus, we get an induced map of \(\Ext\)-groups. 
\end{proof}

By letting \(M=N=\mathcal{C}\), we get the following result. 

\begin{cor}\label{cor:mapRightAlgebra}
  Let \(\mathcal{C}\) be a \(B\)-coring and let \(R\) be the right algebra. Then we have a map
  \begin{equation*}
    \begin{tikzcd}
      \Hh_{\Ca}^*(\mathcal{C}) \ar{r}{} & \HH^*(R|B)
    \end{tikzcd}
  \end{equation*}
  of \(k\)-vector spaces.
\end{cor}

We now give a sufficient condition for the map in Corollary~\ref{cor:mapRightAlgebra} to be an isomorphism. Therefore, we assume from now on that \(\mathcal{C}\) is finitely generated projective as a left \(B\)-module. This is equivalent to assuming that \(\mathcal{C}\) is a right dual coring of its right algebra. We restrict to the categories of \(\mathcal{C}\)-bicomodules that are finitely generated projective as a left \(B\)-module, respectively the category of \(R\)-bimodules that are finitely generated projective as right \(B\)-modules. We denote these categories by \(\mathcal{C}\Bicomod_{\mathbf{lfgp}}\) and \(R\Bimod_{\mathbf{rfgp}}\) respectively. Note that we get a restricted functor
\begin{equation*}
  D = F_{R\big|\mathcal{C}\Bicomod_{\mathbf{lfgp}}} : \mathcal{C}\Bicomod_{\mathbf{lfgp}} \rightarrow R\Bimod_{\mathbf{rfgp}}
\end{equation*}
by Lemma~\ref{lemma:fdb}. 

\begin{lemma}\label{lemma:dualityBbimod}
  There is a duality
  \begin{equation*}
    D_B : B\Bimod_{\mathbf{lfgp}} \rightarrow B\Bimod_{\mathbf{rfgp}}
  \end{equation*}
  given by \(M \mapsto \Hom_{B\text{-}}(M,B)\). 
\end{lemma}

\begin{proof}
  Note that if \(\{f_i,a_i\}\) is a left finite dual basis for \(M\) then \(\{\ev_{a_i}, f_i\}\) a right finite dual basis for \(\Hom_{B\text{-}}(M,B)\) by Lemma~\ref{lemma:fdb}. Therefore, if \(M\) is left finitely generated projective then \(\Hom_{B\text{-}}(M,B)\) is right finitely generated projective. The inverse functor is given by \(\Hom_{\text{-}B}(-,B)\). Therefore, we indeed have a duality. 
\end{proof}

\begin{lemma}\label{lemma:dualityfgp}
  We have that \(D : \mathcal{C}\Bicomod_{\mathbf{lfgp}} \rightarrow R\Bimod_{\mathbf{rfgp}}\) is a duality of categories.
\end{lemma}

\begin{proof}
  We have an inverse functor
  \begin{equation*}
    D^{-1} = \Hom_{\text{-}B}(-,B) : R\Bimod_{\mathbf{rfgp}} \rightarrow \mathcal{C}\Bicomod_{\mathbf{lfgp}}
  \end{equation*}
  where the coaction is given using the right module analog of Lemma~\ref{Lemma:leftdecomp}. Moreover, since finitely generated projective modules are reflexive, see Remark 2.11 in \cite{Lam1999}, we have that
  \begin{equation*}
    D \circ D^{-1} \cong \ide_{R\Bimod_{\mathbf{rfgp}}} \text{ and } D^{-1} \circ D \cong \ide_{\mathcal{C}\Bicomod_{\mathbf{lfgp}}}
  \end{equation*}\qedhere
\end{proof}

\begin{lemma}\label{lemma:injtoproj}
  The following diagram commutes
  \begin{equation}\label{diagram:proj}
    \begin{tikzcd}
      \mathcal{C}\Bicomod_{\mathbf{lfgp}} \ar{r}{D} & R\Bimod_{\mathbf{rfgp}} \\
      B\Bimod_{\mathbf{lfgp}} \ar{u}{\mathcal{C} \otimes_{B} - \otimes_{B} \mathcal{C}} \ar[swap]{r}{D_B} & B\Bimod_{\mathbf{rfgp}} \ar[swap]{u}{R \otimes_{B} - \otimes_{B} R}
    \end{tikzcd}
  \end{equation}
  up to natural isomorphism. Therefore, \(D\) restricts to an equivalence between the category of relative injective bicomodules and the category of relative projective bimodules.   
\end{lemma}

\begin{proof}
  We begin by showing that the diagram commutes on objects. Let \(M\) be a \(B\)-bimodule. Then, going by the upper corner, we have that \(M\) gets mapped to \(\mathcal{C} \otimes_{B} M \otimes_{B} \mathcal{C}\) which then gets mapped to
  \begin{equation*}
    \Hom_{B\text{-}}(\mathcal{C} \otimes_{B} M \otimes_{B} \mathcal{C},B).
  \end{equation*}
  Going by the lower corner, \(M\) gets mapped to \(R \otimes_{B} \Hom_{B\text{-}}(M,B) \otimes_{B} R\). Note that
  \begin{equation*}
    \Hom_{B\text{-}}(\mathcal{C} \otimes_{B} M \otimes_{B} \mathcal{C},B) \cong R \otimes_{B} \Hom_{B\text{-}}(M,B) \otimes_{B} R
  \end{equation*}
  by applying Lemma~\ref{lemma:fdb} twice. Therefore, the diagram commutes on the level of objects. Next we show that it commutes on morphism. Let \(f : M \rightarrow N\) be a \(B\)-bimodule morphism. By the upper corner, \(f\) gets mapped to the morphism \(D(\ide_{\mathcal{C}} \otimes_{B} f \otimes_{B} \ide_{\mathcal{C}})\), whilst by the lower corner, \(f\) gets mapped to \(\ide_{R} \otimes_{B} D(f) \otimes_{B} \ide_R\). As the following diagram commutes
  \begin{equation*}
    \begin{tikzcd}[column sep = 8em]
      D(\mathcal{C} \otimes_{B} N \otimes_{B} \mathcal{C}) \ar{r}{D(\ide_{\mathcal{C}} \otimes_{B} f \otimes_{B} \ide_{\mathcal{C}})} & D(\mathcal{C} \otimes_{B} M \otimes_{B} \mathcal{C}) \\
      R \otimes_{B} D(N) \otimes_{B} R \ar[swap]{r}{\ide_R \otimes_{B} D(f) \otimes_{B} \ide_R} \ar{u}{\gamma} & R \otimes_{B} D(M) \otimes_{B} R \ar[swap]{u}{\gamma}
    \end{tikzcd}
  \end{equation*}
  we have that diagram~(\ref{diagram:proj}) commutes on morphisms as well. 
\end{proof}

\begin{thm}\label{thm:extduality}
  Let \(\mathcal{C}\) be a \(B\)-coring such that \(\mathcal{C}\) is finitely generated as a left \(B\)-module. Then there is a natural isomorphism
  \begin{equation*}
    \Ext^{*}_{(\mathcal{C}\text{-}\mathcal{C}|B\text{-}B)}(M,N) \cong \Ext^*_{(A^e|B^e)}(D(N), D(M))
  \end{equation*}
  for all left finitely generated projective \(\mathcal{C}\)-bicomodules \(M\) and \(N\).  
\end{thm}

\begin{proof}
  Fix two bicomodules \(M\) and \(N\). Let \(I^{\sbu}\) be a relative injective resolution of \(N\). By Proposition~\ref{prop:rightRelExact} and Lemma~\ref{lemma:injtoproj}, we have that \(D(I^{\sbu})\) is a projective resolution of \(D(N)\). By Lemma~\ref{lemma:dualityfgp}, we have that
  \begin{equation*}
    \Hom^{\mathcal{C}\text{-}\mathcal{C}}(M,I^{\sbu}) \cong \Hom_{R-R}(D(I^{\sbu}),D(M)).
  \end{equation*}
  Therefore
  \begin{align*}
    \Ext^{*}_{(\mathcal{C}\text{-}\mathcal{C}|B\text{-}B)}(M,N) \cong H^n(\Hom^{\mathcal{C}\text{-}\mathcal{C}}(M,I^{\sbu})) &\cong H^n(\Hom_{R\text{-}R}(D(I^{\sbu}),D(M))) \\
                                                                                                                             &\cong \Ext^*_{(A^e|B^e)}(D(N), D(M)). \qedhere
  \end{align*}
\end{proof}

By letting \(M=N=\mathcal{C}\), we get the following. 

\begin{cor}
  Let \(\mathcal{C}\) be a \(B\)-coring which is left finitely generated projective. Then
  \begin{equation*}
    \Hh_{\Ca}^*(\mathcal{C}) \cong \HH^*(R|B)
  \end{equation*}
  as \(k\)-vector spaces. 
\end{cor}

To summarise we have the following diagram

\begin{equation*}
  \begin{tikzcd}
    \mathcal{C}\Bicomod \ar{r}{F_R} & R\Bimod \\
    \mathcal{C}\Bicomod_{\mathbf{lfgp}} \ar{r}{D} \ar[hook]{u}{} \ar[shift left]{d}{U} & R\Bimod_{\mathbf{rfgp}} \ar[hook]{u}{} \ar[swap, shift right]{d}{U} \\
    B\Bimod_{\mathbf{lfgp}} \ar{r}{D_B} \ar[shift left]{u}{\text{Ind}} & B\Bimod_{\mathbf{rfgp}} \ar[shift right, swap]{u}{\text{Ind}}
  \end{tikzcd}
\end{equation*}
where \(D\) is a relative exact duality of categories. 

\section{Gerstenhaber structure for relative Hochschild cohomology and Cartier cohomology}
\label{sec:gerst-struct-relat}

In this section, we show that the isomorphism \(\Hh_{\Ca}^*(\mathcal{C}) \cong \HH^*(R|B)\) lifts to the level of \(B_{\infty}\)-algebras on the chain complexes up to taking the opposite, which then induces that this isomorphism is an isomorphism of Gerstenhaber algebras up to taking the opposite.

\hfill

Let \(\mathcal{C}\) be a coring and let \(\CoEndOp(\mathcal{C})\) be the associated coendomorphism operad. We claim that we have a map of operads with multiplication
\begin{equation*}
  R^{\op}(n) : \CoEndOp(\mathcal{C})(n) \rightarrow \EndOp(R^{\op})(n)
\end{equation*}
by
\begin{equation*}
  f \mapsto D_B(f) \circ \gamma
\end{equation*}
where \(\gamma\) is the map given in Lemma~\ref{Lemma:leftdecomp}. For any map \(g : \mathcal{C} \rightarrow \mathcal{C}^{\otimes_B n}\), we define
\begin{equation*}
  g^{i,m} := 1^{\otimes_B (m-i-1)} \otimes_B g \otimes_B 1^{\otimes_B i} : \mathcal{C}^{\otimes_B m} \rightarrow \mathcal{C}^{\otimes_B (m+n-1)}.
\end{equation*}
Dually, for any map \(h : D(\mathcal{C})^{\otimes_B n} \rightarrow D(\mathcal{C})\), we define
\begin{equation*}
  h_{i,m} := 1^{\otimes_B i} \otimes_B h \otimes_B 1^{\otimes_B (m-i-1)} : D(\mathcal{C})^{\otimes_B (m+n-1)} \rightarrow D(\mathcal{C})^{\otimes_B m}.
\end{equation*}
Given an element \(f_1 \otimes_{B} ... \otimes_{B} f_{n} \in D(\mathcal{C})^{\otimes_B n}\) we define
\begin{equation*}
  f_{n \rightarrow 1}(c_{1 \rightarrow n}) := f_m(c_1f_{m-1}(c_2...f_1(c_m)...)).
\end{equation*}
Note that \(\gamma(f_1 \otimes ... \otimes f_{n})(c_1 \otimes ... \otimes c_{n}) = f_{n \rightarrow 1}(c_{1 \rightarrow n})\). 

\begin{lemma}\label{lemma:compAlg}
  Let \(g : \mathcal{C} \rightarrow \mathcal{C}^{\otimes_B n}\) be a map. The following diagram commutes
  \begin{equation*}
    \begin{tikzcd}[column sep = 8em]
      D(\mathcal{C}^{\otimes_B (m+n-1)}) \ar{r}{D(g^{i,m})} & D(\mathcal{C}^{\otimes_B m}) \\
      D(\mathcal{C})^{\otimes_B (m+n-1)} \ar{u}{\gamma} \ar[swap]{r}{[D(g) \circ \gamma]_{i,m}} & D(\mathcal{C})^{\otimes_B m} \ar[swap]{u}{\gamma}
    \end{tikzcd}
  \end{equation*}
  i.e. \(D(g^{i,m}) \circ \gamma = \gamma \circ [D(g) \circ \gamma]_{i,m}\), where by abuse of notation the vertical maps are given by iterated applications of \(\gamma\). 
\end{lemma}

\begin{proof}
  Let \(f_1 \otimes_B ... \otimes_B f_{m+n-1} \in D(\mathcal{C})^{\otimes_B (m+n-1)}\). By applying \(\gamma\) we get the element
  \begin{equation*}
    c_1 \otimes_B ... \otimes_B c_{m+n-1} \mapsto f_{(m+n-1) \rightarrow 1}(c_{1 \rightarrow (m+n-1)})
  \end{equation*}
  in \(D(\mathcal{C}^{\otimes_B (m+n-1)})\). We precompose with \(g^{i,m}\), which maps \(c_1 \otimes_B ... \otimes_B c_m\) to
  \begin{equation*}
   \underbrace{c_1 \otimes_B ... \otimes_B c_{m-i-1}}_{\# m-i-1} \otimes \underbrace{g(c_{m-i})_{\underline{1}} \otimes_B ... \otimes_B g(c_{m-i})_{\underline{n}}}_{\# n} \otimes \underbrace{c_{m-i+1} \otimes ... \otimes c_m}_{\# i}.
  \end{equation*}
Therefore we get that the composition \(\gamma(f_1 \otimes_B ... \otimes_B f_{m+n-1}) \circ g^{i,m}\) is equal to the function which maps \(c_1 \otimes_B ... \otimes_B c_m\) to 
  \begin{equation*}
    f_{(m+n-1) \rightarrow (n+i+1)}(c_{1 \rightarrow (m-i-1)}f_{(n+i) \rightarrow (i+1)}(g(c_{m-i})_{\underline{1} \rightarrow \underline{n}}f_{i \rightarrow 1}(c_{(m-i+1) \rightarrow m}))).
  \end{equation*}
  Conversely, applying \([D(g) \circ \gamma]_{i,m}\) to \(f_1 \otimes_B ... \otimes_B f_{m+n-1}\) we get
  \begin{align*}
    [D(g) \circ \gamma]_{i,m}(f_1 \otimes_B ... \otimes_B f_{m+n-1}) &= f_1 \otimes_B ... \otimes_B f_i \otimes_B D(g)(\gamma(f_{i+1} \otimes_B ... \otimes_B f_{i+n})) \\
                                                                 &\otimes_B f_{i+n+1} \otimes_B ... \otimes_B f_{m+n-1} \\
                                                                 &= f_1 \otimes_B ... \otimes_B f_i \otimes_B \gamma(f_{i+1} \otimes_B ... \otimes_B f_{i+n}) \circ g \\
                                                                 &\otimes f_{i+n+1} \otimes_B ... \otimes_B f_{m+n-1}. 
  \end{align*}
  Lastly we apply \(\gamma\) to get the function which maps \(c_1 \otimes_B ... \otimes_B c_m\) to 
  \begin{equation*}
    f_{(m+n-1) \rightarrow (n+i+1)}(c_{1 \rightarrow (m-i-1)}f_{(n+i) \rightarrow (i+1)}(g(c_{m-i})_{\underline{1} \rightarrow \underline{n}}f_{i \rightarrow 1}(c_{(m-i+1) \rightarrow m}))).
  \end{equation*}
  Therefore the diagram commutes. 
\end{proof}

\begin{prop}\label{prop:endopRight}
  The maps \(\{R^{\op}(n)\}_{n \geq 0}\) form a morphism of operads with multiplication. 
\end{prop}

\begin{proof}
  Let us denote the partial compositions in the coendomorphism operad by \(\circ_i'\). First, let us show that \(R^{\op}(f \circ'_i g) = R^{\op}(f) \circ_i R^{\op}(g)\). By definition
  \begin{align*}
    R^{\op}(f) \circ_i R^{\op}(g) &= (D(f) \circ \gamma) \circ_i (D(g) \circ \gamma) \\
                           &= [D(f) \circ \gamma] \circ [D(g) \circ \gamma]_{i,m}. 
  \end{align*}
  By Lemma~\ref{lemma:compAlg}, we have that 
  \begin{align*}
    R^{\op}(f \circ_i' g) &= D(f) \circ (D(g^{i,m}) \circ \gamma) \\
                          &= D(f) \circ (\gamma \circ [D(g) \circ \gamma]_{i,m})  \\
                          &= [D(f) \circ \gamma] \circ [D(g) \circ \gamma]_{i,m}.
  \end{align*}
  Therefore \(R^{\op}\) commutes with the partial compositions. Note that the multiplication on \(R^{\op}\) is equal to \(R^{\op}(\pi)\), therefore we have that \({R^{\op}}\) commutes with the multiplication.
\end{proof}

\begin{thm}\label{thm:rightalgebra}
  Let \(\mathcal{C}\) be a coring and let \(R^{\op}\) be the opposite of the right algebra for \(\mathcal{C}\). Then there is a strict \(B_{\infty}\)-morphism
  \begin{equation*}
    C_{\Ca}^*(\mathcal{C}) \rightarrow C^*(R^{\op}| B^{\op})
  \end{equation*}
  from the \(B_{\infty}\)-algebra given by the Cartier cochain complex to the \(B_{\infty}\)-algebra given by the relative Hochschild cochain complex. 
\end{thm}

\begin{proof}
  This follows from Theorem~\ref{thm:OperadB}, as the map \(R^{\op}\) of the operads induces a strict morphism of \(B_{\infty}\)-algebras. 
\end{proof}

\begin{cor}\label{cor:B-infty}
  If \(\mathcal{C}\) is finitely generated projective as a left \(B\)-module, we have a strict isomorphism of \(B_{\infty}\)-algebras
  \begin{equation*}
    C_{\Ca}^*(\mathcal{C}) \cong C^*(R^{\op}| B^{\op}).
  \end{equation*}
\end{cor}

\begin{proof}
  Note that for the category of left finitely generated \(B\)-bimodules, \(D\) will be a duality and \(\gamma\) an isomorphism. Therefore, we have that our underlying map is an isomorphism, and thus a \(B_{\infty}\)-isomorphism.
\end{proof}

\begin{cor}
  If \(\mathcal{C}\) is finitely generated projective as a left \(B\)-module, we have an isomorphism of Gerstenhaber algebras
  \begin{equation*}
    \HH_{\Ca}^*(\mathcal{C}) \cong \HH^*(R^{\op}|B^{\op}).
  \end{equation*}
\end{cor}

\begin{proof}
  As taking homology is a functor from the category of \(B_{\infty}\)-algebras to the category of Gerstenhaber algebras, we get an induced map on the level on homology. 
\end{proof}

Similarly to Proposition 6.4 in \cite{ChenLiWang2021} for absolute Hochschild cochains, the following holds for relative Hochschild cochains. As the proof is more or less identical to the proof in the absolute case, we refer the reader to the proof of Proposition 6.4 in \cite{ChenLiWang2021}

\begin{prop}
  There is a \(B_{\infty}\)-isomorphism
  \begin{equation*}
    C^*(R^{\op}|B^{\op}) \rightarrow C^*(R|B)^{\opp}.
  \end{equation*}
\end{prop}

Therefore, by combining with Corollary~\ref{cor:B-infty} and Theorem~\ref{thm:optr}, we get that

\begin{thm}\label{thm:coringtoalg}
  Let \(\mathcal{C}\) be a \(B\)-coring such that \(\mathcal{C}\) is finitely generated projective as a left \(B\)-module. We have an isomorphism of \(B_{\infty}\)-algebras
  \begin{equation*}
    C^*_{\Ca}(\mathcal{C}) \cong C^*(R|B)^{\opp}
  \end{equation*}
  which induces an isomorphism of Gerstenhaber algebras
  \begin{equation*}
    \Hh_{\Ca}^*(\mathcal{C}) \cong \HH^*(A|B)^{\opp}.
  \end{equation*}
\end{thm}

It was shown in \cite{BKK2020} that there is a one-to-one correspondence of quasi-hereditary algebras with homological exact Borel subalgebras with directed corings. For a background on quasi-hereditary algebras and their connection to corings, see \cite{K2017}. Combining this with Theorem~\ref{thm:coringtoalg}, we get the following

\begin{cor}
  Let \(k\) be an algebraically closed field. Then the one-to-one correspondence
  \begin{equation*}
    \begin{tikzcd}
      \{\text{quasi-hereditary algebras } A \text{ with homological exact Borel subalgebras } B\} \ar[leftrightarrow]{d}{1-1} \\
      \{\text{directed } B\text{-corings } \mathcal{C}\}
    \end{tikzcd}
  \end{equation*}
  is compatible with the opposite of the relative Hochschild cohomology and Cartier cohomology as Gerstenhaber algebras. 
\end{cor}

\section{Entwining structures}

We recall the definitions within the theory of entwining structures following Chapter 5 in \cite{CoringsBW} and relate the equivariant cohomology with the relative Hochschild cohomology of the twisted convolution algebra.

\hfill

Let \(k\) be a field and let \(A\) be a \(k\)-algebra and let \(C\) be a \(k\)-coalgebra. We denote the \(k\)-linear dual \(k\)-algebra of \(C\) by \(C^{\vee}\). 

\begin{mydef}
  An \textit{entwining structure} consist of a triple \((A,C,\psi)\), where \(A\) is a \(k\)-algebra, \(C\) a \(k\)-coalgebra and \(\psi : C \otimes A \rightarrow A \otimes C\) satisfying the following four conditions:
  \begin{enumerate}
  \item \(\psi \circ (\ide_C \otimes \mu) = (\mu \otimes \ide_C) \circ (\ide_C \otimes \psi) \circ (\psi \otimes \ide_A)\);
  \item \((\ide_A \otimes \Delta) \circ \psi = (\psi \otimes \ide_C) \circ (\ide_C \otimes \psi) \circ (\Delta \otimes \ide_A)\);
  \item \(\psi \circ (\ide_C \otimes \iota) = \iota \otimes \ide_C\);
  \item \((\ide_A \otimes \varepsilon) \circ \psi = \varepsilon \otimes \ide_A\). 
  \end{enumerate}
  The map \(\psi\) is called an \textit{entwining map}, and we say that \(A\) and \(C\) are \textit{entwined} by \(\psi\). We use the \(\alpha\)\textit{-notation} for the map \(\psi\), where we write \(\psi(c \otimes a) = a_{\alpha} \otimes c^{\alpha}\), suppressing the summation sign for summing over \(\alpha\). 
\end{mydef}

Given an entwining structure, we can naturally associate a coring to it as follows.

\begin{mydef}
  Let \((A,C,\psi)\) be an entwining structure. The \textit{associated coring} is the \(A\)-bimodule \(\mathcal{C} = A \otimes C\) with left action given by \(a'(a \otimes c) = a'a \otimes c\) and right action \((a \otimes c)a' = a\psi(c \otimes a')\) with comultiplication \(\Delta_{\mathcal{C}} = \ide_A \otimes \Delta_C\) and counit \(\varepsilon_{\mathcal{C}} = \ide_A \otimes \varepsilon_C\). An \textit{entwining module} of \((A,C,\psi)\) is a right \(\mathcal{C}\)-comodule.
\end{mydef}

\begin{ex}
  Let \(H\) be a Hopf-algebra. We have an entwining structure \((H,H,\psi)\) given by the map \(\psi : H \otimes H \rightarrow H \otimes H\), defined by \(\psi(h' \otimes h) = h_{\underline{2}} \otimes S(h_{\underline{1}})h'h_{\underline{3}}\). The category of entwining modules is equivalent to the category of Yetter-Drinfeld modules, see 33.6 in \cite{CoringsBW}. 
\end{ex}

\begin{lemma}\label{lemma:rightalgebraent}
  Let \((A,C,\psi)\) be an entwining structure. The right algebra of the associated coring \(\mathcal{C} = A \otimes C\) is isomorphic to the \(\psi\)-\textit{twisted convolution algebra} \(\Hom_{\psi}(C,A)\), where the product is given by \(f * g = \mu_A \circ 1 \otimes f \circ \psi \circ 1 \otimes g \circ \Delta_C\). 
\end{lemma}

For a proof see 32.9 in \cite{CoringsBW}.

\begin{mydef}
  Let \(A,B\) be two \(k\)-algebras and let \(\Phi : B \otimes A \rightarrow A \otimes B\) be a map such that
  \begin{enumerate}
  \item \(\Phi(b \otimes 1_A) = 1_A \otimes b\);
  \item \(\Phi(1_B \otimes a) = a \otimes 1_B\);
  \item \(\Phi(bd \otimes a) = a_{\Phi\varphi} \otimes b_\varphi d_\Phi\);
  \item \(\Phi(b \otimes ac) = a_\Phi c_\varphi \otimes b_{\Phi\varphi}\); 
  \end{enumerate}
  where  \(-_{\Phi\varphi}\) is Sweedler like notation for iterated application of the map \(\Phi\). The \textit{smash product} \(A \#_\Phi B\) is the \(k\)-module \(A \otimes B\) with product
  \begin{equation*}
    (a \# b)(c \# d) = ac_\Phi \# b_\Phi d
  \end{equation*}
  and unit \(1 \# 1\). 
\end{mydef}

\begin{lemma}\label{lemma:smashproduct}
  Let \((C,A,\psi)\) be an entwining structure and let \(\mathcal{C} = A \otimes C\) be the associated coring. Then
  \begin{equation*}
    (C^{\vee})^{\op}\# A \cong \Hom_{\psi}(C,A)
  \end{equation*}
  as algebras. 
\end{lemma}

For a proof, see Proposition 37 in \cite{CaenepeelMilitaaruZhu2002}.

\begin{mydef}
  Let \((A,C,\psi)\) be an entwining structure. The \(\psi\)\textit{-equivariant cohomology} \(\Hh_{\psi-e}^*(C)\) is defined to be the Cartier cohomology of the associated coring. We denote the cochain complex by \(C^*_{\psi-e}(C)\).   
\end{mydef}

As \(A \otimes C \cong A^{\oplus \dim_k C}\) as a left \(A\)-module, we get the following. 

\begin{lemma}\label{lemma:fingenproj}
  Let \((A,C,\psi)\) be an entwining structure. Then \(\mathcal{C} = A \otimes C\) is left finitely generated projective over \(A\) if \(C\) is a finite-dimensional \(k\)-module. 
\end{lemma}

Thus, by combining Lemma~\ref{lemma:rightalgebraent}, Lemma~\ref{lemma:smashproduct}, Lemma~\ref{lemma:fingenproj} and Theorem~\ref{thm:coringtoalg}, we get the following

\begin{thm}\label{thm:equivariantcohomology}
  Let \((A,C,\psi)\) be an entwining structure such that \(C\) is finitely generated projective over \(k\). Then there is an isomorphism of \(B_{\infty}\)-algebras
  \begin{equation*}
    C^*_{\psi-e}(C) \cong C^*(\Hom_{\psi}(C,A)|A)^{\opp} \cong C^*((C^{\vee})^{\op}\# A | A)
  \end{equation*}
  which thus induces an isomorphism of Gerstenhaber algebras
  \begin{equation*}
    \Hh^*_{\psi-e}(C) \cong \HH^*(\Hom_{\psi}(C,A)|A)^{\opp} \cong \HH^*((C^{\vee})^{\op}\# A | A).
  \end{equation*}
\end{thm}

As we have an isomorphism of \(B_{\infty}\)-algebras, we get the following.

\begin{cor}\label{cor:deformation}
  Let \(k\) be a field of characteristic zero. Let \((A,C,\psi)\) be an entwining structure such that \(C\) is finitely generated projective over \(k\). Then
  \begin{equation*}
    \begin{tikzcd}
      \mathbf{MC}((sC^*_{\psi-e}(C) \otimes (t),d,[-,-])) \ar[leftrightarrow]{r}{1:1} & A\text{-relative deformations of } \Hom_{\psi}(C,A) \\
      & A\text{-relative deformations of } (C^{\vee})^{\op}\# A. \ar[leftrightarrow, swap]{u}{1:1} 
    \end{tikzcd}
  \end{equation*}
\end{cor}

\begin{proof}
  As we have an strict \(B_{\infty}\)-isomorphism \(C^*_{\psi-e}(C) \cong C^*(\Hom_{\psi}(C,A)^{\op}|A^{\op})\) by Corollary~\ref{cor:B-infty}, we get an induced isomorphism
  \begin{equation*}
    sC^*_{\psi-e}(C) \otimes (t) \cong sC^*(\Hom_{\psi}(C,A)^{\op}|A^{\op}) \otimes (t)
  \end{equation*}
  of dg Lie algebras. Therefore, we have that
  \begin{equation*}
    \mathbf{MC}(sC^*_{\psi-e}(C) \otimes (t)) \cong \mathbf{MC}(sC^*(\Hom_{\psi}(C,A)^{\op}|A^{\op}) \otimes (t))
  \end{equation*}
  but \(\mathbf{MC}(sC^*(\Hom_{\psi}(C,A)^{\op}|A^{\op}) \otimes (t)) = \mathbf{MC}(sC^*(\Hom_{\psi}(C,A)|A) \otimes (t))\) is in one-to-one correspondence with \(A\)-relative deformations of \(\Hom_{\psi}(C,A)\) by Proposition~\ref{prop:opMC}. As \(\Hom_{\psi}(C,A) \cong (C^{\vee})^{\op} \# A\) by Lemma~\ref{lemma:smashproduct}, the last part follows. 
\end{proof}

Let \((A,C,\psi)\) be the trivial entwining structure. For example, if \(A=C=kG\) of a finite abelian group, then \(\psi_{\mathcal{YD}} = \psi_{\text{triv}}\). Therefore, we get that the smash product \((C^{\vee})^{\op} \# kG\) is isomorphic to \((C^{\vee})^{\op} \otimes A\) with the usual multiplication.

\begin{prop}
  Let \(k\) be an algebraically closed field of arbritrary characteristic and let \(A\) and \(C\) be finite dimensional. Then 
  \begin{equation*}
    \Hh_{\psi_{\text{triv}}\text{-}e}(H) \cong \Hh_{\Ca}^*(C) \otimes A
  \end{equation*}
  of Gerstenhaber algebras. 
\end{prop}

\begin{proof}
  By Theorem~\ref{thm:equivariantcohomology}, there is an isomorphism of Gerstenhaber algebras 
  \begin{equation*}
    \Hh_{\psi-e}(C) \cong \HH^*((C^{\vee})^{\op} \otimes A| k \otimes A)^{\op}.
  \end{equation*}
  By Proposition~\ref{thm:tensorDecomp}, we get that
  \begin{equation*}
    \HH^*((C^{\vee})^{\op} \otimes A|k \otimes A)^{\op} \cong \HH^*((C^{\vee})^{\op})^{\op} \otimes \HH^*(A|A)^{\op} \cong \Hh_{\Ca}^*(C) \otimes Z(A).
  \end{equation*}
  Thus the result follows. 
\end{proof}

Let \(k\) be a field of characteristic \(p > 0\), let \(G\) be a finite abelian group whose order does divide \(p\) and let \(H=kG\). Then \(kG = kC_{p_1^{r_1}} \otimes ... \otimes kC_{p_1^{r_n}} \otimes L\) be the decomposition of \(G\) into cyclic groups and a group whose order does not divide \(p\). Let \((kG,kG,\psi_{\text{triv}})\) be the trivial entwining structure. Then
\begin{equation*}
  \Hh_{\psi_{\text{triv}}\text{-}e}(kG) \cong \HH^*(kC_{p^{r_1}}) \otimes ... \otimes \HH^*(kC_{p^{r_n}}) \otimes \HH^*(L) \otimes kG
\end{equation*}
Thus, following Section 4 in \cite{LeZhou2014}, one can completely describe \(\Hh^*_{\psi_{\text{triv}}\text{-}e}(kG)\) for any abelian group \(G\).

\printbibliography
\end{document}